\documentclass[preprint]{elsarticle}
\usepackage{geometry}                
\usepackage{graphicx}
\usepackage{amssymb,mathrsfs,amsthm,amscd,amsmath}
\usepackage{epstopdf}
\usepackage{xypic}
\usepackage{bm} 
\usepackage[usenames,dvipsnames]{xcolor}
\usepackage{hyperref}

\newtheoremstyle{cited}%
  {3pt}
  {3pt}
  {\itshape}
  {}
  {\bfseries}
  {.}
  {.5em}
  {\thmname{#1} \thmnumber{#2} \thmnote{\normalfont#3}}

\newtheoremstyle{cited2}%
  {3pt}
  {3pt}
  {\upshape}
  {}
  {\bfseries}
  {.}
  {.5em}
  {\thmname{#1} \thmnumber{#2} \thmnote{\normalfont#3}}

\newtheorem{thm}{Theorem}[section]

\newtheorem{lemma}[thm]{Lemma}
\newtheorem{prop}[thm]{Proposition}
\newtheorem{corollary}[thm]{Corollary}

\theoremstyle{definition}
\newtheorem{question}[thm]{Question}

\newtheorem{eg}[thm]{Example}
\newtheorem{defn}[thm]{Definition}

\theoremstyle{remark}
\newtheorem{remark}[thm]{Remark}
\newtheorem{notation}[thm]{Notation}

\theoremstyle{cited}

\newtheorem{citedlemma}[thm]{Lemma}

\theoremstyle{cited2}
\newtheorem{citeddefn}[thm]{Definition}

\theoremstyle{empty}
\newtheorem{duplicate}{Theorem}

\newcommand{\sop}{system of parameters}

\newcommand{\CM}{Cohen-Macaulay}
\newcommand{\fg}{finitely-generated}
\newcommand{\ph}{phantom}

\newcommand{\phex}{\ph\  extension}

\newcommand{\stwoif}{$\text{S}_2$-ification}

\newcommand{\cld}{complete local domain}
\newcommand{\charp}{characteristic $p>0$}

\newcommand{\cmm}{containment module modification}
\newcommand{\pmm}{parameter module modification}
\newcommand{\axioma}{Functoriality Axiom}
\newcommand{\axiomb}{Semi-residuality Axiom}
\newcommand{\axiomc}{Faithfulness Axiom}
\newcommand{\axiomd}{Generalized Colon-Capturing Axiom}

\newcommand{\symm}{\text{Sym}}

\newcommand{\im}{\text{im}}
\newcommand{\Hom}{\text{Hom}}
\newcommand{\Ext}{\text{Ext}}
\newcommand{\epf}{\text{epf}}
\newcommand{\id}{\text{id}}
\newcommand{{\cl}}{\text{cl}}

\newcommand{\ld}{\text{D}}

\newcommand{{\tr}}{\text{tr}}

\newcommand{\fracR}{\text{Frac}}
\newcommand{\Frac}{\text{Frac}}
\newcommand{\syz}{\text{syz}}
\newcommand{\Ann}{\text{Ann}}

\newcommand{\reg}{\text{reg}}
\newcommand{\ord}{\text{ord}}
\newcommand{\depth}{\text{depth}}

\newcommand{\Q}{\mathbb{Q}}
\newcommand{\Z}{\mathbb{Z}}

\newcommand{\N}{\mathbb{N}}

\title{Closure Operations that Induce Big Cohen-Macaulay Modules and Classification of Singularities}
\author[um]{Rebecca~R.G.\corref{cor1}\fnref{fn1}}
\ead{rirg@umich.edu}

\cortext[cor1]{Corresponding author \\ E-mail address: rirg@umich.edu}

\address[um]{Department of Mathematics, University of Michigan, 2074 East Hall, 530 Church Street, Ann Arbor, MI 48109}

\begin{document}

\begin{abstract}
Geoffrey Dietz introduced a set of axioms for a closure operation on a complete local domain $R$ so that the existence of such a closure operation is equivalent to the existence of a big \CM\ module. These closure operations are called Dietz closures. In complete rings of characteristic $p>0$, tight closure and plus closure satisfy the axioms.

We define module closures and discuss their properties. For many of these properties, there is a smallest closure operation satisfying the property. In particular, we discuss properties of big \CM\ module closures, and prove that every Dietz closure is contained in a big \CM\ module closure. Using this result, we show that under mild conditions, a ring $R$ is regular if and only if all Dietz closures on $R$ are trivial. Finally, we show that solid closure in equal characteristic 0, integral closure, and regular closure are not Dietz closures, and that all Dietz closures are contained in liftable integral closure.
\end{abstract}


\maketitle

\section{Introduction}

The question of the existence of big \CM\ modules has motivated many results in commutative algebra. While they are known to exist over rings of equal characteristic \cite{modulemods} and rings of mixed characteristic and dimension at most 3 \cite{heitmann, dim3bigcmalgebras}, it is not known whether they exist over mixed characteristic rings of higher dimension. The existence of big \CM\ modules (or algebras) is also sufficient to imply a large group of equivalent conjectures, including the Direct Summand Conjecture \cite{directsummand}, Monomial Conjecture \cite{directsummand}, and Canonical Element Conjecture \cite{canonicalelement}. The equal characteristic case of these results was achieved using tight closure methods. One obstruction to extending these techniques to rings of mixed characteristic is the lack of such a closure operation.

In \cite{dietz}, Dietz gave a list of axioms for a closure operation such that for a  local domain $R$, the existence of a closure operation satisfying these properties (which we call a Dietz closure) is equivalent to the existence of a big \CM\ module. The closure operation can be used to show that when module modifications (see Definition \ref{pmmdef}) are applied to $R$, the image of 1 in the resulting module is not contained in the image of the maximal ideal of $R$. When $R$ is complete and has \charp, tight closure is a Dietz closure, as are plus closure and solid closure \cite{dietz}. However, Frobenius closure is not a Dietz closure \cite{dietz}.

In Section \ref{closureprops}, we develop some basic properties of closure operations that are used throughout the paper, including properties of big \CM\ module closures (see Definition \ref{moduleclosure}). This is followed in Section \ref{smallestclosures} by a discussion of properties of closure operations for which there is a smallest closure satisfying the property. In particular, any ring that has a Dietz closure has a smallest Dietz closure, as well as a smallest big \CM\ module closure. In certain rings of dimension 2, the smallest big \CM\ module closure comes from the $S_2$-ification of $R$. Studying the smallest Dietz closure or big \CM\ module closure should provide information on the properties of $R$.

We prove:
\begin{duplicate}[Theorem~\ref{dietzinbigcm}]
Let cl be a Dietz closure on a local domain $(R,m)$. Then cl is contained in the module closure ${{\cl}}_B$ for some big \CM\ module $B$, i.e., for any \fg\ $R$-modules $N \subseteq M$, $N_M^{{\cl}} \subseteq N_M^{{\cl}_B}$.
\end{duplicate}
Using this result, we prove:

\begin{duplicate}[Theorem~\ref{regularimpliestrivial}, Theorem~\ref{trivialimpliesregular}]
Suppose that $(R,m)$ is a local domain that has at least one Dietz closure (in particular, it suffices for $R$ to have equal characteristic and any dimension, or mixed characteristic and dimension at most 3). Then $R$ is regular if and only if all Dietz closures on $R$ are trivial.
\end{duplicate}
In the proof of Theorem \ref{trivialimpliesregular}, we see that a particular module of syzygies gives a nontrivial closure operation, which we can compute explicitly. In Section \ref{otherclosureops}, we use these results to compare Dietz closures to better understood closure operations such as integral closure.

In Theorem \ref{notdietz}, we show that integral closure and regular closure are not Dietz closures using a criterion that can be applied more generally. As a corollary of the above theorems, we also conclude that solid closure is not always a Dietz closure for rings of equal characteristic 0. Studying the reasons why certain closure operations are or are not Dietz closures provides more information on the pieces that are needed to get a good enough closure operation in mixed characteristic.

We conclude with a list of further questions in Section \ref{furtherquestions}. Interestingly, we do not know whether there is a largest big \CM\ module closure, as discussed in Section \ref{largestbigcm}.

\section{Background}
\label{background}

In this section we give the necessary definitions and some notation that will be used throughout the paper.

\begin{defn}
Let $(R,m)$ be a local ring. An $R$-module $B$ is a (balanced) \textit{big \CM\ module} for $R$ if every \sop\ for $R$ is a regular sequence on $B$, and $mB \ne B$. Note that $B$ need not be \fg. A \textit{big \CM\ algebra} for $R$ is a big \CM\ module for $R$ that is also an $R$-algebra.
\end{defn}

\begin{defn}
A \textit{closure operation} cl on a ring $R$ is a map $N \to N_M^{{\cl}}$ of submodules $N$ of \fg\ $R$-modules $M$ such that if $N \subseteq N' \subseteq M$ are \fg\ $R$-modules,
\begin{enumerate}
\item (Extension) $N \subseteq N_M^{\cl}$,
\item (Idempotence) $(N_M^{\cl})_M^{\cl}=N_M^{\cl}$, and
\item (Order-Preservation) $N_M^{\cl} \subseteq (N')_M^{\cl}$.
\end{enumerate}
\end{defn}

\begin{defn}
\label{moduleclosure}
Suppose that $S$ is an $R$-module (resp. $R$-algebra). We can define a closure operation ${\cl}_S$ on $R$ by
\[u \in N_M^{{\cl}_S} \text{ if for all } s \in S,\ s \otimes u \in \im(S \otimes N \to S \otimes M)\]
for any $N \subseteq M$ \fg\ $R$-modules and $u \in M$. This is called a \textit{module (resp. algebra) closure}.
\end{defn}

\begin{remark}
If $S$ is an $R$-algebra, $u \in N_M^{{\cl}_S}$ if and only if 
\[1 \otimes u \in \im(S \otimes N \to S \otimes M).\]
\end{remark}

\begin{citeddefn}[{\cite[Definition~2.2]{dietz}}]
\label{phantomDef}
Let $R$ be a ring with a closure operation cl, $M$ a \fg\ $R$-module, and $\alpha:R \to M$ an injective map with cokernel $Q$. We have a short exact sequence 
\[\begin{CD}
0 @>>> R @>{\alpha}>> M @>>> Q @>>> {0.}
\end{CD}\]
Let $\epsilon \in \Ext_R^1(Q,R)$ be the element corresponding to this short exact sequence via the Yoneda correspondence. We say that $\alpha$ is a \textit{cl-\phex\ }if a cocycle representing $\epsilon$ is in
 $P_1^\vee$ is in $\im(P_0^\vee \to P_1^\vee)_{P_1^\vee}^{\cl}$,
where $P_\bullet$ is a projective resolution of $Q$ and $\vee$ denotes $\Hom_R(-,R)$.
\end{citeddefn}

\begin{remark}
This definition is independent of the choice of $P_\bullet$ \cite[Discussion~2.3]{dietz}.

A split map $\alpha:R \to M$ is cl-\ph\ for any closure operation cl: in this case, the cocycle representing $\epsilon$ is in $\im(P_0^\vee \to P_1^\vee)$. We can view cl-\phex s as maps that are ``almost split" with respect to a particular closure operation.
\end{remark}

\begin{notation}
\label{notation1}
We use some notation from \cite{dietz}. Let $R$ be a ring, $M$ a finitely generated $R$-module, and $\alpha:R \to M$ an injective map with cokernel $Q$. Let $e_1=\alpha(1),e_2, \ldots,e_n$ be generators of $M$ such that the images of $e_2,\ldots,e_n$ in $Q$ form a generating set for $Q$. We have a free presentation for $Q$,
\[
\begin{CD}
{R^m} @>{\nu}>> {R^{n-1}} @>{\mu}>> Q @>>> {0,}
\end{CD}
\]
where $\mu$ sends the generators of $R^{n-1}$ to $e_2,\ldots,e_n$ and $\nu$ has matrix $(b_{ij})_{2 \le i \le n, 1 \le j \le m}$ with respect to some basis for $R^m$. We have a corresponding presentation for $M$, 
\[
\begin{CD}
{R^m} @>{\nu_1}>> {R^n} @>{\mu_1}>> {M} @>>> {0,}
\end{CD}
\]
where $\mu_1$ sends the generators of $R^n$ to $e_1,\ldots,e_n$. Using the same basis for $R^m$ as above, $\nu_1$ has matrix $(b_{ij})_{1 \le i \le n, 1 \le j \le m}$ where $b_{1j}e_1+b_{2j}e_2+\ldots+b_{nj}e_n=0$ in $M$ \cite[Discussion 2.4]{dietz}. The top row of $\nu_1$ gives a matrix representation of the map $\phi:R^m \to R$ in the following diagram:

\[
\begin{CD}
0 @>>> R @>{\alpha}>> M @>>> Q @>>> 0 \\
@. @A{\phi}AA @A{\psi}AA @A{\id_Q}AA @AAA \\
 @. {R^m} @>{\nu}>> {R^{n-1}} @>{\mu}>> Q @>>> 0 \\
\end{CD}
\]
\end{notation}

In \cite[Discussion 2.4]{dietz}, Dietz gives an equivalent definition of a \phex\ using the free presentations $M$ and $Q$ given above. While he assumes that $R$ is a \cld\ and that cl satisfies 2 additional properties, these are not needed for all of the results. We restate some of his results in greater generality below.

\begin{citedlemma}[{\cite[Lemma~2.10]{dietz}}]
\label{lemma2.10}
Let $R$ be a ring possessing a closure operation ${\cl}$. Let $M$ be a finitely generated module, and let $\alpha:R \to M$ be an injective map. Let notation be as above. Then $\alpha$ is a cl-\phex\ of $R$ if and only if the vector $(b_{11},\ldots,b_{1m})^{\tr}$ is in $B_{R^m}^{\cl}$, where $B$ is the $R$-span in $R^m$ of the vectors $(b_{i1},\ldots,b_{im})^{\tr}$ for $2 \le i \le n$.
\end{citedlemma}

\begin{citeddefn}[\cite{dietz}]
\label{axioms}
Let $(R,m)$ be a fixed local domain and let $N, M,$ and $W$ be arbitrary finitely generated $R$-modules with $N \subseteq M$. A closure operation cl is called a \textit{Dietz closure} if the following axioms hold:
\begin{enumerate}
\item (Functoriality) Let $f:M \to W$ be a homomorphism. Then $f(N_M^{\cl}) \subseteq f(N)_W^{\cl}$.
\item (Semi-residuality) If $N_M^{\cl}=N$, then $0_{M/N}^{\cl}=0$.
\item (Faithfulness) The ideal $m$ is closed in $R$.
\item (Generalized Colon-Capturing) Let $x_1,\ldots,x_{k+1}$ be a partial system of parameters for $R$, and let $J=(x_1,\ldots,x_k)$. Suppose that there exists a surjective homomorphism $f:M \to R/J$ and $v \in M$ such that $f(v)=x_{k+1}+J$. Then $(Rv)_M^{\cl} \cap \ker f \subseteq (Jv)_M^{\cl}$.
\end{enumerate}
\end{citeddefn}

\begin{remark}
The axioms originally included the assumption that $0^{\cl}_R=0$, but this is implied by the other axioms \cite{newdietz}.

A closure operation on any ring $R$ can satisfy the \axioma, the \axiomb, or both. A closure operation on any local ring $R$ can satisfy the \axiomc.
\end{remark}

The proof of the next lemma requires $Q$ to have a minimal generating set, so we assume that $R$ is local for this generalization of \cite[Lemma~2.11]{dietz}:

\begin{lemma}
\label{lemma2.11}
Let $(R,m)$ be a local ring possessing a closure operation ${\cl}$ that satisfies the \axioma, the \axiomb, and the \axiomc. If $M$ is a finitely generated $R$-module such that $\alpha:R \to M$ is cl-phantom, then $\alpha(1) \not\in mM$.
\end{lemma}

\begin{citeddefn}[{\cite[Discussion 5.15]{hhpaper}}]
\label{pmmdef}
Let $R$ be local and $M$ an $R$-module. A \textit{\pmm} of $M$ is a map
\[M \to M'=\frac{M \oplus Rf_1 \oplus \ldots \oplus Rf_k}{R(u \oplus x_1f_1 \oplus \ldots \oplus x_kf_k)},\]
where $x_1,\ldots,x_{k+1}$ is part of a system of parameters for $R$ and $u_1,\ldots,u_k,u$ are elements of $M$ such that
\[x_{k+1}u=x_1u_1+\ldots+x_ku_k.\]
\end{citeddefn}

\begin{remark}
Dietz proves in \cite{dietz} that a local domain $R$ has a Dietz closure if and only if it has a big \CM\ module. In his proof that a Dietz closure can be used to construct a big \CM\ module, one could replace the \axiomd\ with any axiom that implies that given a cl-\phex\ $\alpha:R \to M$ and a \pmm\ $M \to M'$, the map $R \to M'$ is still a cl-\phex. However, we do not know of a good candidate to replace the Axiom.
\end{remark}

\section{Properties of closure operations}
\label{closureprops}

We list some properties of closure operations that will be needed later.

\begin{lemma}
\label{lemma1.2}
Let $R$ be a ring possessing a closure operation cl. In the following, $N,N',$ and $N_i \subseteq M_i$ are all $R$-submodules of the finitely generated $R$-module $M$.
\begin{itemize}
\item[(a)] Suppose that cl satisfies the \axioma\ and the \axiomb. Let $N' \subseteq N \subseteq M$. Then $u \in N_M^{\cl}$ if and only if $u+N' \in (N/N')_{M/N'}^{\cl}$.
\item[(b)] Suppose that cl satisfies the \axioma, $\cal{I}$ is a finite set, $N=\bigoplus_{i \in \cal{I}} N_i$, and $M=\bigoplus_{i \in \cal{I}} M_i$. Then $N_M^{\cl}=\bigoplus_{i \in \cal{I}} (N_i)_{M_i}^{\cl}$.
\item[(c)] Let $\cal{I}$ be any set. If $N_i \subseteq M$ for all $i \in \cal{I}$, then $\left(\bigcap_{i \in \cal{I}} N_i\right)_M^{\cl} \subseteq \bigcap_{i \in \cal{I}} (N_i)_{M_i}^{\cl}$.
\item[(d)] Let $\cal{I}$ be any set. If $N_i$ is cl-closed in $M$ for all $i \in \cal{I}$, then $\bigcap_{i \in \cal{I}} N_i$ is cl-closed in $M$.
\item[(e)] $(N_1+N_2)_M^{\cl}=\left((N_1)_M^{\cl}+(N_2)_M^{\cl}\right)_M^{\cl}$.
\item[(f)] Suppose that cl satisfies the \axioma. Let $N \subseteq N' \subseteq M$. Then $N_{N'}^{\cl} \subseteq N_M^{\cl}$.
\item[(g)] Suppose that $R$ is a domain, cl satisfies the \axioma, $0^{\cl}_R=0$, and $M$ is a torsion-free $R$-module. Then $0^{\cl}_M=0$.
\item[(h)] Suppose that $(R,m)$ is local and cl satisfies the \axioma, the \axiomb, and the \axiomc. Then $N_M^{\cl} \subseteq N+mM$.
\end{itemize}
\end{lemma}

\begin{proof}
Parts (a) to (e) are proved in \cite[Lemma~1.2]{dietz}.

For part (f), let $f:N' \to M$ be the inclusion map. Then by the \axioma,
\[N_{N'}^{\cl} = f(N_{N'}^{\cl}) \subseteq f(N)_M^{\cl}=N_M^{\cl}.\]  

For part (g), notice that $M \hookrightarrow R^s$ for some $s>0$. By part (f), $0^{\cl}_M \subseteq 0^{\cl}_{R^s}$. By part (b), $0^{\cl}_{R^s}=\bigoplus 0^{\cl}_R=0$.

For part (h), we first prove that for $F$ a \fg\ free module, $(mF)_F^{\cl}=mF$. By part (a), this is equivalent to $0_{F/mF}^{\cl}=0$. Let $u \in 0^{\cl}_{F/mF}$ be nonzero. Then there exists a map $\phi:F/mF \to R/m$ with $\phi(u) \ne 0$. By the \axioma, $\phi(u) \in 0^{\cl}_{R/m}=0$ (since $m^{\cl}_R=m$), which is a contradiction. Hence $0^{\cl}_{F/mF}=0$.

By part (a), it suffices to show that $0^{\cl}_M \subseteq mM$. Let 
\[\begin{CD}
F_1 @>>> F_0 @>{\pi}>> M @>>> 0
\end{CD}\]
be part of a minimal free resolution of $M$. Then $\im(F_1) \subseteq mF_0$. This implies that $\im(F_1)^{\cl}_{F_0} \subseteq (mF_0)_{F_0}^{\cl}=mF_0$. By part (a), $0^{\cl}_M=\pi(\im(F_1)^{\cl}_{F_0})$.
We have
\[0^{\cl}_M = \pi(\im(F_1)^{\cl}_{F_0}) \subseteq \pi(mF_0) = m\pi(F_0)=mM,\]
as desired.
\end{proof}

\begin{lemma}
\label{moduleclosurelemma}
Let $R$ be a ring and $S$ an $R$-module or $R$-algebra. Then ${\cl}_S$ satisfies the \axioma\ and the \axiomb. Hence ${\cl}_S$ has properties (a)-(f) of Lemma \ref{lemma1.2}. Further, for $N \subseteq M$ finitely generated $R$-modules, ${\cl}_S$ satisfies
 \[I^{{\cl}_S}N_M^{{\cl}_S} \subseteq (IN)_M^{{\cl}_S}\] for all $I \subseteq R$.
 In particular, $yN_M^{{\cl}_S} \subseteq (yN)_M^{{\cl}_S}$ for all $y \in R$. 
\end{lemma}

\begin{remark}
If $R$ is a domain, then this Lemma implies that ${\cl}_S$ is semi-prime as in \cite{epstein}.
\end{remark}

\begin{proof}
First we show that ${\cl}_S$ satisfies the \axioma\ and the \axiomb. 
Suppose that $N \subseteq M$ and $W$ are finitely generated $R$-modules, and $f:M \to W$ is an $R$-module homomorphism. Let $u \in N_M^{{\cl}_S}$. We will show that $f(u) \in f(N)_W^{{\cl}_S}$. For every $s \in S$, $s \otimes u \in \im(S \otimes N \to S \otimes M)$. Applying $\id_S \otimes_R f$, we get $s \otimes f(u) \in \im(S \otimes f(N) \to S \otimes W)$ for every $s \in S$. So ${\cl}_S$ satisfies the \axioma.

Suppose $N^{{\cl}_S}_M=N$. We will show that $0^{{\cl}_S}_{M/N}=0$. Let $\bar{u} \in 0^{{\cl}_S}_{M/N}$. Then for every $s \in S$, $s \otimes \bar{u}=0$ in $S \otimes M/N$. Since $S \otimes \_$ is right exact, $S \otimes M/N \cong (S \otimes M)/(S \otimes N)$. Thus $s \otimes u \in \im(S \otimes N \to S \otimes M)$. Since this holds for every $s \in S$, $u \in N^{{\cl}_S}_M=N$. Thus $\bar{u}=0$ in $M/N$. So ${\cl}_S$ satisfies the \axiomb.

Now we prove that 
 \[I^{{\cl}_S}N_M^{{\cl}_S} \subseteq (IN)_M^{{\cl}_S}\] 
 for all $I \subseteq R$. Suppose that $u \in N_M^{{\cl}_S}$ and $y \in I^{{\cl}_S}$.  Then for every $s \in S$, 
\[s \otimes u \in \im(S \otimes N \to S \otimes M),\]
 and $ys \in IS$. In particular, for every $s \in S$, 
 \[s \otimes yu=ys \otimes u = i_1(s_1 \otimes u)+i_2(s_2 \otimes u)+\ldots+i_n(s_n \otimes u)\]
  for some $i_1,\ldots,i_n \in I$, $s_1,\ldots,s_n \in S$. But each $i_j(s_j \otimes u)=s_j \otimes i_ju \in \im(S \otimes IN \to S \otimes M)$. Hence $yu \in (IN)_M^{{\cl}_S}$.
  
The last statement follows, because
\[yN_M^{{\cl}_S} \subseteq (y)^{{\cl}_S}N_M^{{\cl}_S} \subseteq (yN)_M^{{\cl}_S}\]
by the previous statement.
\end{proof}

The following lemma allows us to generalize the idea of an algebra closure.

\begin{lemma}
\label{directedfamilyclosure}
Let $\cal{S}$ be a directed family of $R$-algebras. We can define a closure operation ${\cl}_{\cal{S}}$ by $u \in N_M^{{\cl}_{\cal{S}}}$ if for some $S \in \cal{S}$, $u \in N_M^{{\cl}_S}$.
\end{lemma}

\begin{proof}
To see that $N_M^{{\cl}_{\cal{S}}}$ is a submodule of $M$, let $u,v \in N_M^{{\cl}_{\cal{S}}}$. It is clear that for any $r \in R$, $ru \in N_M^{{\cl}_{\cal{S}}}$.To see that $u+v \in N_M^{{\cl}_{\cal{S}}}$, note that there is some $S,S' \in \cal{S}$ such that $u \in N_M^{{\cl}_S}$ and $v \in N_M^{{\cl}_{S'}}$. Since $\cal{S}$ is a directed family, there is some $T \in \cal{S}$ such that $S,S'$ both map to $T$. We will have $1 \otimes u,1\otimes v \in \im(T \otimes N \to T \otimes M)$, so $1 \otimes (u+v) \in \im(T \otimes N \to T \otimes M)$. Hence $u+v \in N_M^{{\cl}_T} \subseteq N_M^{{\cl}_{\cal{S}}}$.

The extension and order-preservation properties of a closure operation are not difficult to prove. We prove the idempotence property. Let $u \in (N_M^{{\cl}_{\cal{S}}})_M^{{\cl}_{\cal{S}}}$. 
Then for some $S \in \cal{S}$, $1 \otimes u \in \im(S \otimes N_M^{{\cl}_{\cal{S}}} \to S \otimes M)$, say $1 \otimes u=\sum_{i=1}^n s_i \otimes u_i$ with the $u_i \in N_M^{{\cl}_{\cal{S}}}$. For each $i$, there is some $S_i \in \cal{S}$ such that $u_i \in N_M^{{\cl}_{S_i}}$. There is some $T \in \cal{S}$ such that each $S_i$ maps to $T$. Hence $1 \otimes u \in \im(T \otimes N \to T \otimes M)$.
\end{proof}

\begin{prop}
Let cl be a closure operation that commutes with finite direct sums (in particular, it is enough to assume that cl satisfies the \axioma). Suppose the map $R \to M$ that sends $1 \mapsto u$ is cl-phantom, as is the map $R \to N$ that sends $1 \mapsto v$. Then the map $f:R \to (M \oplus N)/(u \oplus -v)$ that sends $1 \mapsto (u,0)=(0,v)$ is cl-phantom, too. Further, any \phex\ $R \to Q$ that factors through both $M$ and $N$ factors through $(M \oplus N)/(u \oplus -v)$ as well.
\end{prop}

Note: If $f$ split, we would have $M=R \oplus M_0$, $N=R \oplus N_0$, and $(M \oplus N)/(u \oplus -v)=R \oplus (M_0 \oplus N_0)$. 

\begin{proof}
The last statement is automatic from the definition of a push-out. The cokernel $f$ is the direct sum of the cokernels of the maps $R \to M$ and $R \to N$, and the direct sum of free resolutions $P_\bullet$ and $P'_\bullet$, respectively, of these cokernels gives us a free resolution of the cokernel of $f$. If $\phi:P_1 \to R$ and $\phi':P_1' \to R$ are maps induced by the identity map on the cokernels, then the hypothesis tells us that 
\[\phi \in \left(\im(\Hom(P_0,R) \to \Hom(P_1,R))\right)^{\cl}_{\Hom(P_1,R)}\]
and
\[\phi' \in \left(\im(\Hom(P'_0,R) \to \Hom(P'_1,R))\right)^{\cl}_{\Hom(P'_1,R)}.\]
Since cl commutes with direct sums, we get
\[\phi \oplus \phi' \in \left(\im(\Hom(P_0 \oplus P'_0,R) \to \Hom(P_1 \oplus P'_1,R))\right)^{\cl}_{\Hom(P_1 \oplus P'_1,R)},\]
as desired.
\end{proof}

\begin{prop}
\label{containment}
Let $S$ and $T$ be $R$-modules such that for each $t \in T$, there is a map $S \to T$ whose image contains $t$. Then ${\cl}_S \subseteq {\cl}_T$, i.e., for any \fg\ $R$-modules $N \subseteq M$, $N_M^{{\cl}_S} \subseteq N_M^{{\cl}_T}$. 
\end{prop}

\begin{proof}
Suppose that $N \subseteq M$ are \fg\ $R$-modules, and that $u \in N_M^{{\cl}_S}$. We will show that $u \in N_M^{{\cl}_T}$. Since $u \in N_M^{{\cl}_S}$, for each $s \in S$, $s \otimes u \in \im(S \otimes N \to S \otimes M)$. Let $t \in T$. Then there is some map $f:S \to T$ whose image contains $t$, say $s' \mapsto t$. There is some element $y$ of $S \otimes N$ that maps to $s' \otimes u$ in $S \otimes M$. The image $(f \otimes \id)(y)$ of $y$ in $T \otimes N$ maps to $t \otimes u$ in $T \otimes M$, by the commutativity of the following diagram:

\[\begin{CD}
{S \otimes N} @>>> {S \otimes M} \\
@V{f \otimes \id}VV @V{f \otimes \id}VV \\
{T \otimes N} @>>> {T \otimes M} \\
\end{CD}\]

Hence $t \otimes u \in \im(T \otimes N \to T \otimes M)$ for every $t \in T$, which implies that $u \in N_M^{{\cl}_T}$.
\end{proof}

\begin{notation} 
We refer to the intersection of two closure operations cl and ${\cl}'$, as defined in \cite{epstein}. Let $N \subseteq M$ be \fg\ $R$-modules. We say that
\[u \in N_M^{{\cl} \cap {\cl}'} \text{ if } u \in N_M^{\cl} \cap N_M^{{\cl}'}.\]
\end{notation}

\begin{prop}
\label{intersection}
Let $S$ and $T$ be $R$-modules. Then ${\cl}_{S \oplus T}={\cl}_S \cap {\cl}_T$.
\end{prop}

\begin{proof}
Suppose that $N \subseteq M$ are \fg\ $R$-modules, and $u \in N_M^{{\cl}_{S \oplus T}}$. Then for each $(s,t) \in S \oplus T$, 
\[(s,t) \otimes u \in \im((S \oplus T) \otimes N \to (S \oplus T) \otimes M)=\im(S \otimes N \to S \otimes M) \oplus \im(T \otimes N \to T \otimes M).\]
So $s \otimes u$ is in the first image, and $t \otimes u$ is in the second. Thus $u \in N_M^{{\cl}_S} \cap N_M^{{\cl}_T}$. If $u \in N_M^{{\cl}_S} \cap N_M^{{\cl}_T}$, then for each $s \in S$, $s \otimes u \in \im(S \otimes N \to S \otimes M)$ and for each $t \in T$, $t \otimes u \in \im(T \otimes N \to T \otimes M)$. Hence $(s,t) \otimes u \in \im((S \oplus T) \otimes N \to (S \oplus T) \otimes M).$
\end{proof}

\subsection{Properties of Big \CM\ Module Closures}
\label{bigcmclosureprops}

We give several useful properties of big \CM\ module closures.

\begin{defn}
\label{strongcc}
Let $cl$ be a closure operation on a ring $R$. 
\begin{enumerate}
\item We say that cl satisfies \textit{colon-capturing} if for every partial \sop\ $x_1,\ldots,x_{k+1}$ on $R$, 
\[(x_1,\ldots,x_k):x_{k+1} \subseteq (x_1,\ldots,x_k)^{\cl}.\]
\item We say that cl satisfies \textit{strong colon-capturing, version A}, if for every partial \sop\ $x_1,\ldots,x_k$ on $R$,
\[(x_1^t,x_2,\ldots,x_k)^{\cl} : x_1^a \subseteq (x_1^{t-a},x_2,\ldots,x_k)^{\cl}\]
for all $a<t$.
\item We say that cl satisfies \textit{strong colon-capturing, version B}, if for every partial \sop\ $x_1,\ldots,x_{k+1}$ on $R$,
\[(x_1,\ldots,x_k)^{\cl}:x_{k+1} \subseteq (x_1,\ldots,x_k)^{\cl}.\]
This is a stronger condition than colon-capturing.
\end{enumerate}
\end{defn}

\begin{prop}
\label{strongccprop}
Let $B$ be a big \CM\ module over a local domain $R$. Then the module closure ${\cl}_B$ satisfies strong colon-capturing, version A.
\end{prop}

\begin{proof}
Let $x_1,\ldots,x_k$ be a partial \sop\ on $R$.
Suppose that $a<t$, and that $u \in (x_1^t,x_2,\ldots,x_k)^{\cl}:x_1^a$. In other words, for each $b \in B$,
\[ux_1^ab \in (x_1^t,\ldots,x_k)B,\]
say $ux_1^ab=x_1^tb_1+x_2b_2+\ldots+x_kb_k$. Then $x_1^a(ub-x_1^{t-a}b_1) \in (x_2,\ldots,x_k)B$. Since $B$ is a big \CM\ module, this implies that $ub-x_1^{t-a}b_1 \in (x_2,\ldots,x_k)B$. Hence $ub \in (x_1^{t-a},x_2,\ldots,x_k)B$. Since this holds for each $b \in B$, $u \in (x_1^{t-a},x_2,\ldots,x_k)^{{\cl}_B}$.
\end{proof}

\begin{prop}
\label{otherstrongcc}
Let $B$ be a big \CM\ module over a local domain $R$. Then ${\cl}_B$ satisfies strong colon-capturing, version B. As a consequence, ${\cl}_B$ satisfies colon-capturing. 
\end{prop}

\begin{proof}
Let $x_1,\ldots,x_{k+1}$ be a partial \sop\ on $R$.
Suppose that $v \in R$ such that $vx_{k+1} \in (x_1,\ldots,x_k)^{{\cl}_B}$. Then for each $b \in B$, $vx_{k+1}b \in (x_1,\ldots,x_k)B$. Equivalently, $x_{k+1}(vb) \in (x_1,\ldots,x_k)B$. Since $x_1,\ldots,x_{k+1}$ form part of a \sop\ on $R$, they form a regular sequence on $B$. Hence $vb \in (x_1,\ldots,x_k)B$. As we proved this for an arbitrary $b \in B$, $v \in (x_1,\ldots,x_k)^{{\cl}_B}$, as desired.
\end{proof}

\section{Smallest Closures}
\label{smallestclosures}

\subsection{Intersection Stable Properties}

Given a set $\{{\cl}_\lambda\}_{\lambda \in \Lambda}$ of closure operations, their intersection $\bigcap_{\lambda \in \lambda} {\cl}_\lambda$ is also a closure operation \cite[Construction~3.1.3]{epstein}.

\begin{defn}
Given a property P of a closure operation, we call P \textit{intersection stable} if whenever ${\cl}_\lambda$ satisfies P for every $\lambda \in \Lambda$, $\bigcap_{\lambda \in \Lambda} {\cl}_\lambda$ also satisfies P.
\end{defn}

The following lemma is immediate:

\begin{lemma}
Suppose that P is an intersection stable property of a closure operation and that $R$ has a closure operation satisfying P. Then $R$ has a smallest closure operation satisfying P.
\end{lemma}

\begin{thm}
The \axioma\ is intersection stable. The \axiomb\ is intersection stable on sets of closures that satisfy the \axioma. When $R$ is local, the \axiomc\ and the \axiomd\ are intersection stable.
\end{thm}

\begin{proof}
Let $\{{\cl}_\lambda\}_{\lambda \in \Lambda}$ be a family of closure operations, and 
\[{\cl}=\bigcap_{\lambda \in \Lambda} {\cl}_\lambda.\]
If each ${\cl}_\lambda$ satisfies the \axioma, $f:M \to W$ is an $R$-module map, and $N \subseteq M$ is a submodule, then $f(N_M^{\cl}) \subseteq f(N_M^{{\cl}_\lambda}) \subseteq f(N)_W^{{\cl}_\lambda}$ for each $\lambda$.
Thus $f(N_M^{\cl}) \subseteq \bigcap_\lambda f(N)_W^{{\cl}_\lambda}=f(N)_W^{\cl}$, as desired.

Suppose that $N_M^{\cl}=N$, and that for each $\lambda$, ${\cl}_\lambda$ satisfies the \axioma\ and the \axiomb. We will show that $0_{M/N}^{\cl}=0$. Suppose that $\bar{u} \in 0_{M/N}^{\cl}$. Then for each $\lambda$, $\bar{u} \in 0_{M/N}^{{\cl}_\lambda}$.  By Lemma \ref{lemma1.2}, $u \in N_M^{{\cl}_\lambda}$ if and only if $\bar{u} \in 0_{M/N}^{{\cl}_\lambda}$. Hence $u \in N_M^{{\cl}_\lambda}$ for each $\lambda$, which implies that $u \in N_M^{\cl}=N$. Thus $\bar{u}=0$, and so cl satisfies the \axiomb.

It is clear that the \axiomc\ is intersection stable.

Suppose that ${\cl}_\lambda$ satisfies the \axiomd\ for each $\lambda$ and that $x_1,\ldots,x_{k+1}$ is part of a \sop\ for $R$, $J=(x_1,\ldots,x_k)$, and $f:M \twoheadrightarrow R/J$ such that there is some $v \in M$ with $f(v)=x_{k+1}+J$. We need to show that $(Rv)_M^{\cl} \cap \ker(f) \subseteq (Jv)_M^{\cl}$. Since $(Rv)_M^{\cl} \cap \ker(f) \subseteq (Rv)_M^{{\cl}_\lambda} \cap \ker(f) \subseteq (Jv)_M^{{\cl}_\lambda}$ for each $\lambda$, the \axiomd\ holds for cl.
\end{proof}

\begin{corollary}
If a local domain $R$ has a Dietz closure, then it has a smallest Dietz closure.
\end{corollary}

In the case of a \CM\ ring, the smallest Dietz closure is the trivial closure. However, we do not know what it looks like in more generality. 

\begin{remark}
Colon-capturing is a useful property for a closure operation to have, but it is not enough on its own for our purposes. For example, the closure $N_M^{\cl}=M$ captures colons, but is too large to be useful.
\end{remark}

\begin{lemma}
\label{ccintersectionstable}
Colon-capturing is an intersection stable property.
\end{lemma}

\begin{proof}
This is immediate from Definition \ref{strongcc}.
\end{proof}

\begin{lemma}
Strong colon-capturing, version A, as in Definition \ref{strongcc} is intersection stable.
\end{lemma}

\begin{proof}
To see this, notice that if $x_1,\ldots,x_k$, $t$, and $a$ are as in the definition of strong colon-capturing, version A, then 
\[
(x_1^t,x_2,\ldots,x_k)^{\cl} :_R x_1^a 
\subseteq (x_1^t,x_2,\ldots,x_k)^{{\cl}_\lambda} :_R x_1^a 
\subseteq (x_1^{t-a},x_2,\ldots,x_k)^{{\cl}_\lambda}
\]
 for each $\lambda$. Hence $(x_1^t,x_2,\ldots,x_k)^{\cl}:_R x_1^a \subseteq (x_1^{t-a},x_2,\ldots,x_k)^{\cl}$.
\end{proof}

\begin{remark}
A similar proof works for strong colon-capturing, version B.
\end{remark}

If cl is defined on a category of rings, then we would like to find the smallest closure operation as above (if any such exist) that captures colons and also satisfies the following property: 

\begin{defn}
A closure operation satisfies \textit{persistence for change of rings} if whenever $R \to S$ is a morphism in this category, and $N \subseteq M$ are finitely generated $R$-modules, then $\im(S \otimes_R N_M^{\cl} \to S \otimes_R M) \subseteq (\im(S \otimes_R N \to S \otimes_R M))_{S \otimes_R M}^{\cl}$.
\end{defn}

\begin{remark}
\label{persistenceremark}
Tight closure satisfies both persistence for change of rings and colon-capturing when $R$ is a complete local domain \cite{smoothbasechange}. 

The trivial closure always satisfies persistence for change of rings, but captures colons if and only if $R$ is \CM.
\end{remark}

\begin{prop}
\label{persistenceintersectionstable}
Persistence for change of rings is an intersection stable property.
\end{prop}

\begin{proof}
Suppose that ${\cl}_\lambda$ are closure operations, each defined on all rings in the category, that are persistent for change of rings. Let ${\cl}=\bigcap_{\lambda \in \Lambda} {\cl}_\lambda$. We will show that cl is persistent for change of rings. Let $R \to S$ be a morphism in the category, and suppose that $u \in N_M^{\cl}$. Our goal is to show that $1 \otimes u \in (\im(S \otimes_R N \to S \otimes_R M))^{\cl}_{S \otimes_R M}$. By definition of cl, $u \in N_M^{{\cl}_\lambda}$ for every $\lambda \in \Lambda$. Since each ${\cl}_\lambda$ is persistent with change of rings, this implies that 
\[1 \otimes u \in (\im(S \otimes_R N \to S \otimes_R M))^{{\cl}_\lambda}_{S \otimes_R M}\]
for every $\lambda \in \Lambda$.
Hence $1 \otimes u \in (\im(S \otimes_R N \to S \otimes_R M))^{\cl}_{S \otimes_R M}.$
\end{proof}

\begin{corollary}
The category of all complete local domains has a smallest persistent closure operation that captures colons.
\end{corollary}

\begin{proof}
This follows immediately from Lemma \ref{ccintersectionstable}, Remark \ref{persistenceremark}, and Proposition \ref{persistenceintersectionstable}.
\end{proof}

\begin{question}
When $R$ is a complete local domain that is not \CM, what is the smallest persistent closure operation that captures colons?
\end{question}

\subsection{Smallest Big \CM\ Module Closure}

Given a big \CM\ module $B$ over a local domain $R$, we get a module closure ${\cl}_B$. In \cite{dietz}, Dietz proves that ${\cl}_B$ is a Dietz closure. We can define a new closure operation by intersecting all of these closures. Since the property of being a Dietz closure is intersection stable, this is also a Dietz closure. As we prove below, it is also a big \CM\ module closure.

\begin{prop}
\label{smallestbigcmmoduleclosure}
Let $R$ be a local domain, and let $B$ be a big \CM\ module constructed using the method of \cite{dietz}. If $B'$ is any big \CM\ $R$-module, ${\cl}_B \subseteq {\cl}_{B'}$. As a consequence, ${\cl}_B$ is the smallest big \CM\ module closure on $R$.
\end{prop}

\begin{proof} 
Let $B$ be a big \CM\ module constructed as above, and $B'$ an arbitrary big \CM\ module. Then for each map $R \to B'$, we can construct a map $B \to B'$ that takes the image of 1 in $B$ to the image of 1 in $B'$ via the given map $R \to B'$. To get this map, we start with the map $R \to B'$. If we already have maps from $M_0=R, M_1,\ldots,M_t$ to $B'$, we extend the map to $M_{t+1}$ as follows: 
\[M_{t+1}=(M \oplus Rf_1 \oplus \ldots \oplus Rf_k)/(u \oplus x_1f_1 \oplus \ldots x_kf_k)\]
for some $u \in M_t$ and partial \sop\ $x_1,\ldots,x_k$ for $R$ such that 
\[x_{k+1}u=x_1m_1+\ldots+x_km_k\]
is a bad relation in $M_t$. Since $B'$ is a big \CM\ module, the image of $u$ in $B'$ under the map already constructed is in $(x_1,\ldots,x_k)B'$, say $u=x_1b_1+\ldots+x_kb_k$ with $b_1,\ldots,b_k \in B'$. We extend our map $M_t \to B'$ to a map from $M_{t+1}$ to $B'$ by sending $f_i \mapsto b_i$. 
Take the direct limit of this system of maps $M_t \to B'$ as $t \to \infty$ to get the desired map $B \to B'$. Since we can start with any map $R \to B'$, every element of $B'$ is in the image of a map constructed this way. Hence Proposition \ref{containment} implies that ${\cl}_B \subseteq {\cl}_{B'}$.
\end{proof}

In certain rings of dimension 2, we know more about the smallest big \CM\ module closure. 

\begin{citeddefn}[\cite{s2ifications}]
For $R$ a local domain, the \textit{$S_2$-ification} of $R$ is the unique smallest extension of $R$ in its fraction field that satisfies Serre's condition $S_2$, if such a ring exists. When it exists, it can be constructed by adding to $R$ all elements $f \in \Frac(R)$ such that some height 2 ideal of $R$ multiplies $f$ into $R$.
\end{citeddefn}

\begin{prop}
Let $R$ be a local domain of dimension 2 that has an $S_2$-ification $S$. Then the module closure ${\cl}_S$ is the smallest big \CM\ module closure on $R$.
\end{prop}

\begin{proof}
Let $B$ be a big \CM\ module constructed by the method of \cite{dietz}, so that ${\cl}_B$ is the smallest big \CM\ module closure on $R$. Since $S$ is \CM\ when $R$ has dimension 2, we know that ${\cl}_B \subseteq {\cl}_S$. By Proposition \ref{containment}, it is enough to show that for any map $R \to B$, $1 \mapsto u$, we have a map $S \to B$ whose image contains $u$. To do this, we need to extend the map from $R$ to $S$ by defining it on elements $f \in \fracR(R)$ such that some height 2 ideal of $R$ multiplies $f$ into $R$. Let $f$ be such an element. Since $\dim(R)=2$, there is some \sop\ $x,y$ for $R$ such that $xf,yf \in R$. Then the map is already defined on $xf, yf$, say $xf \mapsto v$, $yf \mapsto w$. The element $xyf$ must map to $yv$, but also must map to $xw$, so $yv=xw$. Since $x,y$ is a regular sequence on $B$, $v=xv_0$ and $w=yw_0$ for some $v_0,w_0 \in B$. Then $xyv_0=yv=xw$, so $w=yv_0$. Hence $yv_0=yw_0$, which implies that $v_0=w_0$. Thus $f \mapsto v_0$ is a well-defined extension of the map $R \to B$. Further, $1_S$ maps to $u$, so this is the map we need to see that ${\cl}_S \subseteq {\cl}_B$.
\end{proof}

\begin{eg}
Let $R=k[[x^4,x^3y,xy^3,y^4]]$. The \stwoif\ $S$ of $R$ must contain $x^2y^2$, since $x^4(x^2y^2)=(x^3y)^2 \in R$ and $y^4(x^2y^2)=(xy^3)^2 \in R$. In fact, $S$ is the subring $k[[x^4,x^3y,x^2y^2,xy^3,y^4]]$ of $k[[x,y]]$. 
 Since $(x^3y)^2=x^4(x^2y^2)$ in $S$, $(x^3y)^2 \in (x^4)_R^{{\cl}_S}$. Similarly, $(xy^3)^2 \in (y^4)_R^{{\cl}_S}$. Hence $(x^3y)^2 \in (x^4)^{{\cl}}_R$ and $(xy^3)^2 \in (y^4)^{\cl}_R$ for every Dietz closure cl on $R$.
\end{eg}

\subsection{Smallest module closure containing another closure}

Given a closure operation cl on $R$, we can construct the smallest module closure containing cl. This will be used later on to prove that every Dietz closure is contained in a big \CM\ module closure. To construct the smallest module closure containing a given closure, we use a second type of module modification.

\begin{defn}
\label{cmm}
Let cl be a closure operation on $R$, $G \subseteq R^s$ a submodule of a \fg\ free $R$-module generated by
\[e_1=(e_{11},\ldots,e_{1s}),\ldots,e_k=(e_{k1},\ldots,e_{ks}),\]
 and let $v=(v_1,\ldots,v_s) \in G_{R^s}^{\cl} -G$. A \textit{\cmm} of an $R$-module $M$ relative to an element $x \in M$ is a map
\[M \to M'=\frac{M \oplus Rf_1 \oplus \ldots \oplus Rf_k}{R(v_1x \oplus e_{11}f_1 \oplus \ldots \oplus e_{k1}f_k,\ldots,v_sx \oplus e_{1s}f_1 \oplus \ldots \oplus e_{ks}f_k)}.\]
\end{defn}

\begin{prop}
\label{cmmdirectlimit}
Let $R$ be a ring, $W$ an $R$-module, and cl a closure operation on $R$ satisfying the \axioma\ and the \axiomb.
Then there is an $R$-module $S$ with a map $\phi:W \to S$ such that ${\cl} \subseteq {\cl}_S$, and for any $R$-module $T$ such that  ${\cl} \subseteq {\cl}_T$ and any map $\psi:W \to T$, we have a map $\gamma:S \to T$ such that $\psi=\gamma \circ \phi$.
\end{prop}

\begin{proof}
To create such an $S$, we apply \cmm s to \fg\ submodules of $W$. First, we show that we have a direct limit system of \cmm s. Given a finite set of modules $G_1,\ldots,G_t$  with $G_i \subseteq R^{s_i}$, and for each $i,$ a finite set of elements $v_{i1},v_{i2},\ldots,v_{i\ell_i} \in (G_i)_{R^{s_i}}^{\cl}-{G_i}$, we can apply finitely many \cmm s to a finitely-generated submodule $W_0 \subseteq W$ to get a module $W_1$ such that for each $1 \le i \le t$ and $1 \le j \le \ell_i$, 
\[\im(v_{ij} \otimes W_0 \to R^{s_i} \otimes W_1) \subseteq \im(G_i \otimes W_1 \to R^{s_i} \otimes W_1).\] 
Then we apply finitely many \cmm s to $W_1$, forcing 
\[\im(v_{ij} \otimes W_1 \to R^{s_i} \otimes W_2) \subseteq \im(G_i \otimes W_2 \to R^{s_i} \otimes W_2)\]
for all $i,j$.
Repeating this process infinitely many times, we get a module $W_\infty$ that is the direct limit of the $W_r$ and such that
\[\im(v_{ij} \otimes W_\infty \to R^s \otimes W_\infty) \subseteq \im(G_i \otimes W_\infty \to R^s \otimes W_\infty)\]
for all $i,j$. We have a map $W_0 \to W_\infty$ since each \cmm\ comes with a map from $W_0$.

Consider all finite sets $\cal{G}=\{G_1,\ldots,G_t,v_{11},v_{12},\ldots,v_{1\ell_1},v_{21},v_{22},\ldots,v_{t\ell_t}\}$ with $G_i \subseteq R^{s_i}$ and finitely many elements $v_{i1},\ldots,v_{i\ell_i} \in (G_i)^{\cl}_{R^s}-G_i$ for each $1 \le i \le t$, and also all \fg\ submodules $W_0$ of $W$. Suppose that $\cal{G} \subseteq \cal{G'}$ are two such sets, that $W_0 \subseteq W_0'$ are \fg\ submodules of $W$, and that $W_\infty$ and $W'_\infty$ are corresponding direct limit modules constructed from $W_0$ using $\cal{G}$ and from $W_0'$ using $\cal{G'}$, respectively. We build a map $W_\infty \to W'_\infty$, starting with the map $W_0 \subseteq W_0' \to W'_\infty$. 

It suffices to demonstrate that the map can be extended to a single \cmm. Let $P$ be an intermediate module in the direct limit system of $W_\infty$ with a map $P \to W'_\infty$, $v=v_{ij} \in \cal{G}$ for some $i,j$, $e_1,\ldots,e_k$ be the generators of $G=G_i$, and $x \in Q$ as in Definition \ref{cmm}. We need to specify the images of $f_1,\ldots,f_k$ in $W'_\infty$. Since $v \otimes W'_\infty \subseteq G \otimes W'_\infty$, $vx=e_1w_1+e_2w_2+\ldots+e_kw_k$ for some $w_1,\ldots,w_k \in W'_\infty$. Then the map that sends $f_i \mapsto w_i$ is a well-defined extension of the map $P \to W'_\infty$. Hence we have a map $W_\infty \to W'_\infty$ for any $\cal{G} \subseteq \cal{G'}$.

The $W_\infty$ form a partially ordered set via $W_\infty \le W'_\infty$ if the corresponding finite sets satisfy $\cal{G} \subseteq \cal{G'}$ and $W_0 \subseteq W_0'$. This is a directed set, using the maps $W_\infty \to W'_\infty$ we constructed above. Let $S$ be the direct limit. By the set-up above, we have a well-defined map $\phi:W \to S$. 
We are now done proving that for submodules $G$ of \fg\ free $R$-modules $R^s$, $G^{\cl}_{R^s} \subseteq G^{{\cl}_S}_{R^s}$.

Suppose that $N \subseteq M$ are arbitrary \fg\ $R$-modules. We will show that $N^{\cl}_M \subseteq N^{{\cl}_S}_M$. There is some $s$ for which $M/N \cong R^s/G$, where $G$ is a submodule of $R^s$. 
Let $u \in N_M^{\cl}$. By Lemma \ref{lemma1.2}, part (a), $\bar{u} \in 0^{\cl}_{M/N} \cong 0^{\cl}_{R^s/G}$. Applying the Lemma again, any lift $v$ of $\im(\bar{u})$ to $R^s$ is in $G^{\cl}_{R^s}$, which is contained in $G^{{\cl}_S}_{R^s}$ by the previous paragraph. Applying the Lemma twice more, we get $\bar{u} \in 0^{{\cl}_S}_{M/N}$, which implies that $u \in N_M^{{\cl}_S}$.



Now suppose that $T$ is an $R$-module such that ${\cl} \subseteq {\cl}_T$, and we have a map $\psi:W \to T$. Let $\phi:W \to S$ be as above. For any intermediate module $P$ in the direct limit system of $S$, let $\phi_P$ be the corresponding map $W \to P$. Suppose that we have a map $\gamma_P:P \to T$ such that $\psi=\gamma_P \circ \phi_P$. We demonstrate how to extend the map to a map $\gamma_{P'}:P' \to T$ such that $\psi=\gamma_{P'} \circ \phi_{P'}$ when $P'$ is a \cmm\ of $P$. We have:
\[P \to P'=\frac{P \oplus Rf_1 \oplus \ldots \oplus Rf_k}{R(v_1x \oplus e_{11}f_1 \oplus \ldots \oplus e_{k1}f_k,\ldots,v_sx \oplus e_{1s}f_1 \oplus \ldots \oplus e_{ks}f_k)},\]
where $x \in P$, and $v,e_1,\ldots,e_k$ are as in Definition \ref{cmm}. We need to specify the images of the $f_i$. Since ${\cl} \subseteq {\cl}_T$, $vx \in (e_1,\ldots,e_k)T$, say $vx=e_1t_1+\ldots+e_kt_k$. Then sending $f_i \mapsto t_i$ gives us a well-defined extension of $\gamma_P$ such that $\psi=\gamma_{P'} \circ \phi_{P'}$. Since $S$ is a direct limit of such \cmm s, we get a map $\gamma:S \to T$ such that $\psi=\gamma \circ \phi$.
\end{proof}

\begin{thm}
Let $R$ be a ring and cl a closure operation on $R$ satisfying the \axioma\ and the \axiomb. Then if we set $W=R$ and construct a module $S$ as in Proposition \ref{cmmdirectlimit}, ${\cl}_S$ is the smallest module closure containing cl, i.e., if $T$ is any $R$-module such that ${\cl} \subseteq {\cl}_T$, we have ${\cl}_S \subseteq {\cl}_T$. In particular, if cl is a module closure, then ${\cl}={\cl}_S$ (conversely, if cl is not a module closure, then ${\cl} \subsetneq {\cl}_S$).
\end{thm}

\begin{proof}
By Proposition \ref{cmmdirectlimit}, for every $R$-module map $R \to T$, we have a map $S \to T$ that agrees with the original map on the image of $R$. So for every element $t \in T$, we have a map $S \to T$ whose image contains $t$. By Proposition \ref{containment}, this implies that ${\cl}_S \subseteq {\cl}_T$.
\end{proof}

\section{A connection between Dietz closures and singularities}
\label{dietzandsing}

In this section, we show that for any local domain $R$ that has a Dietz closure, $R$ is regular if and only if all Dietz closures on $R$ are trivial. First, we prove a result on the relationship between general Dietz closures and big \CM\ module closures.

\begin{thm}
\label{dietzinbigcm}
Let cl be a Dietz closure on a local domain $(R,m)$. Then cl is contained in ${\cl}_B$ for some big \CM\ module $B$.
\end{thm}

\begin{proof}
Let cl be a Dietz closure on $R$. To construct $B$, we use both \pmm s and \cmm s. First, we construct a big \CM\ module $S_1$ using \pmm s as in \cite{dietz}. We apply \cmm s to $S_1$ as in Proposition \ref{cmmdirectlimit} to get a module $S_2$ such that ${\cl} \subseteq {\cl}_{S_2}$ and a map $S_1 \to S_2$, and then we use \pmm s to construct an $R$-module $S_3$ such that every system of parameters on $R$ is a regular sequence on $S_3$ and a map $S_2 \to S_3$.
We repeat these two constructions countably many times, getting maps
\[R=S_0 \to S_1 \to S_2 \to S_3 \to \ldots \] 
The direct limit $B$ is an $R$-module such that ${\cl} \subseteq {\cl}_B$ and every \sop\ on $R$ is a regular sequence on $B$. We need to show that $\im(1) \not\in mB$ when we apply the map $R \to B$ that is the direct limit of the maps $R \to S_i$.

We follow the proof of \cite[Proposition 3.7]{bigcmalgapps}. If $\im(1) \in mB$, then there is a \fg\ $R$-module $P$ with $1 \in mP$ such that $P$ maps to $B$. 

Claim: There is an $R$-module $W$ constructed from $R$ by taking finitely many module modifications (of either or both types) such that the map $P \to B$ passes through $W$.

\begin{proof}[Proof of Claim] Given any \fg\ $R$-module $P$ with a map $P \to B$, there is some $i>0$ for which $\im(P) \subseteq S_i$. Then there is also a finite sequence of \cmm s and \pmm s of $S_{i-1}$ giving a module $W_{i-1}$ such that the map $P \to B$ passes through $W_{i-1}$. We use induction on the value of $i$. 
If $i=1$, then the result is immediate. Suppose the result holds for $i=1,2,\ldots,k-1$, and let $S$ be a module gotten from $S_{k-1}$ by applying a finite sequence of module modifications, such that $\im(P) \subseteq S$. By induction, there is an $R$-module $W_{k-1}$ that is constructed from $R$ by taking finitely many module modifications, and such that $\im(P \cap S_{k-1}) \subseteq W_{k-1}$. Any element of $P$ not in $S_{k-1}$ must come from one of the module modifications applied to $S_{i-1}$ to get $S$. So when we apply the same sequence of module modifications to $W_{k-1}$, we get an $R$-module $W_k$ that is constructed by applying finitely many module modifications to $R$ and such that $\im(P) \subseteq W_k$.
\end{proof}

Further, if we apply any finite sequence of module modifications to $R$ to get a module $W$, we have a map $W \to B$, constructed in the same way as the maps $W_\infty \to W'_\infty$ in the proof of Proposition \ref{cmmdirectlimit} and the maps $M_t \to B'$ in the proof of Proposition \ref{smallestbigcmmoduleclosure}. Therefore, $\im(1) \in mB$ if and only if $\im(1) \in mW$, where $W$ is an $R$-module obtained by applying finitely many module modifications to $R$. We will show that we cannot have $\im(1) \in mW$. To do this, we show that if we have a cl-phantom map $R \to M$, and we apply a single module modification to $M$ to get $M'$, the resulting map $R \to M'$ is cl-phantom. Hence $\im(1) \not\in mM'$.

Assume $\alpha:R \to M$ is a \phex\ of $R$. If we apply a \pmm\ to $M$, we know that the resulting map $\alpha':R \to M'$ is phantom by \cite{dietz}. In the following Lemma, we show that $\alpha':R \to M'$ is phantom when we apply a \cmm\ to $M$. Hence by Lemma \ref{lemma2.11}, $\alpha'(1) \not\in mM'$. This guarantees that in the limit, $mB \ne B$.
\end{proof}

\begin{lemma}
Suppose that $(R,m)$ is a local domain and cl is a Dietz closure on $R$ that satisfies the \axioma\ and the \axiomb, and such that $0^{\cl}_R=0$. Suppose that $\alpha:R \to M$ is a cl-\phex, and let $M'$ be a \cmm\ of $M$. Then $\alpha':R \to M'$ is a cl-\phex.
\end{lemma}

\begin{proof}
 Let $v=(v_1,\ldots,v_s) \in G_{R^s}^{\cl}-G$ for some nonzero submodule $G \subseteq R^s$ (as $0^{\cl}_{R^s}=0$ by assumption), and let $x \in M$. Let $u$ be the image of 1 in $M$. Taking a single module modification, we get 
\[ M'=\frac{M \oplus Rf_1 \oplus \ldots \oplus Rf_k}{R\left(v_1x \oplus e_{11}f_1 \oplus \ldots \oplus e_{k1}f_k, \ldots, v_sx \oplus e_{1s}f_1 \oplus \ldots \oplus e_{ks}f_k\right)}.\]

First, we need to show that the composite map $\alpha':R \to M \to M'$ is injective. Let $F=\Frac(R)$. Then $F \to F \otimes_R M$ is injective, and it suffices to show that $F \to F \otimes M'$ is injective, i.e. that it is nonzero (if $R \to M'$ were not injective, applying $F \otimes$ would preserve this). We claim that $v \in \im(F \otimes G \to F^s)$. To see that this is true, notice that by Lemma \ref{lemma1.2}, $0^{\cl}_{R^s/G}$ is contained in the torsion part of $R^s/G$. Hence $v \in G_{R^s}^{\cl}$ implies that $\bar{v}$ is a torsion element of $R^s/G$. Hence $\bar{v}=0$ in $F^s/(F \otimes G)$, which implies that $v \in \im(F \otimes G \to F^s)$. Then the relations we kill to get $F \otimes M'$ already hold in $F \otimes M$, so there is a retraction $F \otimes M' \to F \otimes M$. This implies that $F \otimes M \to F \otimes M'$ is injective, and so $F \to F \otimes M'$ is injective, as desired.

\begin{remark}
In the special case $s=1$, we can show that the map $M \to M'$ sending each element $y \mapsto y\oplus 0 \oplus \ldots \oplus 0$ is injective. If $y \mapsto 0$, then $y \oplus 0 \oplus \ldots \oplus 0=r(vx \oplus r_1f_1 \oplus \ldots \oplus r_kf_k)$ in $M \oplus Rf_1 \oplus \ldots \oplus Rf_k$, for some $r \in R$. We may assume without loss of generality that some $r_i$ is nonzero, say $r_1$. Then $rr_1f_1=0$, so $rr_1=0$. Since $R$ is a domain, $r=0$. So $y=rvx=0$.
\end{remark}

Following Notation \ref{notation1} and \cite[Discussion~2.4]{dietz}, pick a generating set $w_1,\ldots,w_n$ for $M$ such that $w_1=u$ and $w_n=x$. Then the images of $w_2,\ldots,w_n$ form a generating set for $Q$. Let
\[\begin{CD}
R^m @>{\nu}>> R^{n-1} @>{\mu}>> Q @>>> 0
\end{CD}\]
be a free presentation of $Q$, where $\mu$ sends the generators of $R^{n-1}$ to $w_2,\ldots,w_n$, respectively. We can choose a basis for $R^m$ such that $\nu$ is given by the $(n-1) \times m$ matrix $(b_{ij})_{2 \le i \le n, 1 \le j \le m}$. As in \cite{dietz}, we construct the diagram
\[\begin{CD}
R^m @>{\nu_1}>> R^n @>{\mu_1}>> M @>>> 0 \\
@VV{\id}V @VV{\pi}V @VVV \\
R^m @>{\nu}>> R^{n-1} @>{\mu}>> Q @>>> 0, \\
\end{CD}\]
where $\pi$ kills the first generator of $R^n$ and the rows are exact. The map $\mu_1$ sends the generators of $R^n$ to $w_1,\ldots,w_n$, respectively, and $\nu_1$ has matrix $(b_{ij})_{1 \le i \le n, 1 \le j \le m}$ with respect to the same basis for $R^m$ used to give $\nu$.

Now we construct corresponding resolutions for $M'$ and $Q'$. $M'$ has $k$ new generators and $s$ new relations, as does $Q'$, so we get the following diagram:

\[\begin{CD}
R^{m+s} @>{\nu'_1}>> R^{n+k} @>{\mu'_1}>> M' @>>> 0 \\
@VV{\id}V @VV{\pi}V @VVV \\
R^{m+s} @>{\nu'}>> R^{n-1+k} @>{\mu'}>> Q' @>>> 0 \\
\end{CD}\]

The maps $\mu'$ and $\mu_1'$ take the generators of $Q'$ and $M'$ to $\overline{w_2},\ldots,\overline{w_n},f_1,\ldots,f_k$ and $w_1,\ldots,w_n,f_1,\ldots,f_k$, respectively. The map $\pi$ kills the first generator of $R^{n+k}$. The map $\nu'_1$ can be given by the matrix
\[
\left(
\begin{array}{c | c}
 & 0 \\
\mbox{\Huge $\nu_1$} & \vdots \\
 & 0 \\
  & v \\
\hline \\
 & e_1 \\
\mbox{\Huge $0$} & \vdots \\
 & e_k \\
\end{array}
\right),
\]

and $\nu'$ is this matrix with the top row removed.

The rows of this diagram are exact. We demonstrate the exactness at $R^{n+k}$. To see that $\mu'_1 \circ \nu'_1=0$, we observe that $\mu_1 \circ \nu_1=0$, and for all $i$, $v_ix+e_{1i}f_1+\ldots+e_{ki}f_k=0$ in $M'$. 
To see that $\ker(\mu'_1) \subseteq \im(\nu'_1)$, suppose that $\mu'_1(a_1,\ldots,a_{n+k})^{\tr}=0$. Then 
\[\begin{aligned}
a_1w_1+\ldots+a_nw_n+a_{n+1}f_1+\ldots+a_{n+k}f_{k}&=r_1(v_1x+e_{11}f_1+\ldots+e_{k1}f_k) \\ &+r_2(v_2x+e_{12}f_1+\ldots+e_{k2}f_k) \\
+\ldots&+r_s(v_sx+e_{1s}f_1+\ldots+e_{ks}f_k) \\
\end{aligned}\]
 in $M \oplus Rf_1 \oplus \ldots \oplus Rf_k$, for some $r_1,\ldots,r_s \in R$. So 
 \[a_1w_1+\ldots+a_{n-1}w_{n-1}+(a_n-\sum_{i=1}^sr_sv)x=0,\]
  and 
  \[a_{n+i}=\sum_{j=1}^s r_je_{ij}\]
   for $1 \le i \le k$. This implies that $(a_1,\ldots,a_{n-1},a_n-\sum_{i=1}^sr_sv,0,\ldots,0)^{\tr}$ is in the image of the first $m$ columns of $\nu'_1$, and $(0,\ldots,0,\sum_{i=1}^sr_sv,a_{n+1},\ldots,a_{n_k})^{\tr} $ is in the image of the last $s$ columns of $\nu'_1$. Hence $(a_1,\ldots,a_{n+k})^{\tr}$ is in the image of $\nu'_1$, as desired.

By Lemma \ref{lemma2.10}, $\alpha'$ is \ph\  if and only if the top row of $\nu'_1$ is in the cl-closure of the span of the other rows. Denote the top row of $\nu_1$ by $\bm{x}$, the bottom row by $\bm{y}$, and the span of the middle rows by $H$.
Then $\alpha'$ is \ph\  if and only if 
\[\bm{x} \oplus 0 \in (R(\bm{y} \oplus v) + (H \oplus \bm{0}) + (\bm{0} \oplus G))_{R^{m+s}}^{\cl}.\]
But since $\alpha$ is phantom, $\bm{x} \in (R\bm{y}+H)_{R^m}^{\cl}$. Hence $\bm{x} \oplus \bm{0} \in (R\bm{y}+H)_{R^{m}}^{\cl} \oplus \bm{0}$, and we have
\[\begin{aligned}
(R\bm{y}+H)_{R^{m}}^{\cl} \oplus \bm{0} &=(R\bm{y}+H)_{R^m}^{\cl} \oplus 0_{R^s}^{\cl} \\
&=((R\bm{y}+H) \oplus \bm{0})_{R^{m+s}}^{\cl} \\
&=((R\bm{y} \oplus \bm{0})+(H \oplus \bm{0}))_{R^{m+s}}^{\cl} \\ 
\end{aligned}\]
We want to show that this is contained in $(R(\bm{y} \oplus v)+(H \oplus \bm{0})+(\bm{0} \oplus G))_{R^{m+s}}^{\cl}$. We have
\[\begin{aligned}
 (R\bm{y} \oplus 0)+(H \oplus \bm{0})  &\subseteq R(\bm{y} \oplus v)+(H \oplus \bm{0})+(\bm{0} \oplus G^{\cl}_{R^s}) \\
&=(R\bm{y} \oplus v)+(H \oplus \bm{0})+((\bm{0} \oplus G))_{R^{m+s}}^{\cl} \\
&\subseteq (R(\bm{y} \oplus v)+(H \oplus \bm{0}))^{\cl}_{R^{m+s}}+(\bm{0} \oplus G)_{R^{m+s}}^{\cl}. \\
\end{aligned}\]
Thus 
\[\begin{aligned}
((R\bm{y} \oplus 0)+(H \oplus \bm{0}))_{R^{m+s}}^{\cl} &\subseteq \left(\left(R(\bm{y} \oplus v)+(H \oplus \bm{0})\right)^{\cl}_{R^{m+s}}+(\bm{0} \oplus G)_{R^{m+s}}^{\cl}\right)_{R^{m+s}}^{\cl} \\
&=(R(\bm{y} \oplus v)+(H \oplus \bm{0}) + (\bm{0} \oplus G))^{\cl}_{R^{m+s}}
\end{aligned}\]
by Lemma \ref{lemma1.2}.
Therefore, $\alpha'$ is phantom.
\end{proof}

It turns out that the closure operation ${\cl}_B$ from Theorem \ref{dietzinbigcm} is the smallest big \CM-module closure containing cl, the initial Dietz closure.

\begin{lemma}
\label{smallestbigcmcontainingdietz}
Let notation be as in Theorem \ref{dietzinbigcm}. Given a big \CM-module $B'$ such that ${\cl} \subseteq {\cl}_{B'}$, ${\cl}_B \subseteq {\cl}_{B'}$.
\end{lemma}

\begin{proof}
For any map $R \to B'$, we construct a map $B \to B'$. We already know from the proof of Proposition \ref{smallestbigcmmoduleclosure} how to extend the map $M \to B'$ to a map $M' \to B'$, where $M'$ is a \pmm\ of $M$. We need to know how to extend the map when 
\[M'= \frac{M \oplus Rf_1 \oplus \ldots \oplus Rf_k}{R\left(v_1x \oplus e_{11}f_1 \oplus \ldots \oplus e_{k1}f_k, \ldots, v_sx \oplus e_{1s}f_1 \oplus \ldots \oplus e_{ks}f_k\right)}\]
is a \cmm\ of $M$.
Since $(v_1,\ldots,v_s) \in G_{R^s}^{\cl}$, for each $b' \in B'$, $(v_1,\ldots,v_s) \otimes b' \in \im(G \otimes B' \to R^s \otimes B')$. In particular, for each $1 \le i \le s$, $v_ix=e_{1i}b_1+e_{2i}b_2+\ldots+e_{ki}b_k$ where $b_1,\ldots,b_k \in B'$. Define the map $M' \to B'$ by sending $f_i \mapsto b_i$.

Now for every map $R \to B'$ sending $1 \mapsto u$, we have a map $B \to B'$ whose image contains $u$. So by Proposition \ref{containment}, ${\cl}_B \subseteq {\cl}_{B'}$.
\end{proof}

\begin{question}
\begin{enumerate}
\item Are all Dietz closures big \CM\ module closures, or any kind of module closure? If not, is there a nice way of characterizing the difference between Dietz closures that are big \CM\ module closures and those that are not?
\item If we use only \cmm s as in Proposition \ref{cmmdirectlimit}, are there useful hypotheses that guarantee that the constructed module $S$ is a big \CM\ module?
\end{enumerate}
\end{question}

We use the following definition in our proof that Dietz closures are trivial on regular rings.

\begin{defn}
Given a closure operation cl, a ring $R$ is \textit{weakly cl-regular} if for $N \subseteq M$ finitely generated $R$-modules, $N^{\cl}_M=N$.
\end{defn}

\begin{remark}
It is equivalent to say that $I^{\cl}_R=I$ for all ideals $I$ of $R$. This follows from an argument in \cite{tightclosure}.
\end{remark}

\begin{prop}
\label{trivialonregular}
Let cl be a closure operation on a regular local ring $(R,m)$ that satisfies
\begin{enumerate}
\item strong colon-capturing, version A,
\item $m^{\cl}=m$, and
\item if $N' \subseteq N \subseteq M$ are finitely-generated $R$-modules, then $(N')_N^{\cl} \subseteq (N')_M^{\cl}$.
\end{enumerate}
Then $R$ is weakly cl-regular.
\end{prop}

\begin{proof}
Let $N \subseteq M$ be finitely-generated $R$-modules, and let $x_1,\ldots,x_d$ be regular parameters for $R$ (i.e., $(x_1,\ldots,x_d)=m$). Since $N=\bigcap_s (N+m^sM)$, by Lemma \ref{lemma1.2} it suffices to show that $N+m^sM$ is cl-closed in $M$ for each $s$. Fix a value of $s$. By the same Lemma, we may replace $M$ by $M/(N+m^sM)$ and show that $0$ is cl-closed in this module instead. Since $M$ now has finite length, for some $t$, $I_t=(x_1^{t+1},x_2^{t+1},\ldots,x_d^{t+1})$ kills $M$, and so $M$ is an $R/I_t$-module. Now $I_t$ is $m$-primary, so $R/I_t$ is 0-dimensional. Additionally, $R$ is regular and $x_1,\ldots,x_d$ form a system of parameters, so $R/I_t$ is Gorenstein. Hence $R/I_t$ is injective as a module over itself and is also the only indecomposable injective $R/I_t$-module. This implies that $M \hookrightarrow (R/I_t)^h$ for some $h \ge 0$. Now it suffices to show that $I_t$ is cl-closed in $R$, as then 0 is cl-closed in $(R/I_t)^h$. Since $0 \subseteq M \subseteq (R/I_t)^h$, this implies that $0_M^{\cl} \subseteq 0_{(R/I_t)^h}^{\cl}=0$.

We show that $I_t$ is cl-closed in $R$ for all $t$. Let $x=x_1 x_2 \cdots x_d$. Since $(x_1,\ldots,x_d)=m$, $1$ generates the socle in $R/I_0=R/m$. Then $x^t$ generates the socle in $R/I_t$ for $t \ge 1$. So if $I_t$ is not cl-closed, we must have $x^t \in (I_t)^{\cl}_R$. Thus it suffices to show that $x^t \not\in (I_t)_R^{\cl}$.

Suppose that $x^t \in (I_t)^{\cl}_R$. Then 
\[x_1^t(x_2^t \cdots x_d^t) \in (x_1^{t+1},\ldots,x_d^{t+1})^{\cl}_R.\]
By hypothesis (1) on cl, 
\[x_2^t \cdots x_d^t \in (x_1,x_2^{t+1},\ldots,x_d^{t+1})^{\cl}_R.\]
Using this hypothesis again, 
\[x_3^t \cdots x_d^t \in (x_1,x_2,x_3^{t+1},\ldots,x_d^{t+1})^{\cl}_R.\]
Continuing in this manner, we see that 
\[x_d^t \in (x_1,x_2,\ldots,x_{d-1},x_d^{t+1})^{\cl}_R,\]
and taking one more step, $1 \in (x_1,\ldots,x_d)^{\cl}_R$. However, $m^{\cl}_R=m$, so this is a contradiction. Therefore, $(I_t)_R^{\cl}=I_t$ for all $t$, which finishes the proof that $N_M^{\cl}=N$ for all submodules $N$ of \fg\ $R$-modules $M$.
\end{proof}

\begin{thm}
\label{regularimpliestrivial}
Dietz closures are trivial on regular local rings.
\end{thm}

\begin{proof}
Earlier, we showed that any Dietz closure is contained in a big \CM\ module closure and that big \CM\ module closures satisfy strong colon-capturing. Since they are Dietz closures, they satisfy the other two properties required to use Proposition \ref{trivialonregular}.
Therefore, Dietz closures are trivial on regular rings.
\end{proof}

It is also possible to show that big \CM\ module closures are trivial on regular rings by noting that a big \CM\ module $B$ over a regular ring is faithfully flat \cite{faithfullyflat}, so that ideals and submodules of \fg\ modules are ``contracted" from $B$. 

\begin{thm} 
\label{trivialimpliesregular}
Suppose that $(R,m,K)$ is a local domain that has at least one Dietz closure (in particular, it suffices for $R$ to have equal characteristic and any dimension, or mixed characteristic and dimension at most 3), and that all Dietz closures on $R$ are trivial. Then $R$ is regular.
\end{thm}

\begin{proof}
Since $R$ has a big \CM\ module $B$ that gives a trivial Dietz closure ${\cl}_B$, $R$ is \CM. We show that $R$ is also approximately Gorenstein. If $\dim(R) \ge 2$, then $\depth(R) \ge 2$, so this follows from \cite{cyclicpurity}. If $\dim(R)=0$, then $R$ is a field, which is approximately Gorenstein. If $\dim(R)=1$, then the integral closure $S$ of $R$ is a big \CM\ algebra for $R$. Let $b/a \in S$. We have $b \in (a)^{{\cl}_S}$, but ${\cl}_S$ must be trivial on $R$, so $b \in (a)$. Hence $S=R$, and so $R$ is normal. By \cite{cyclicpurity}, $R$ is approximately Gorenstein. 

Let $I_1 \supseteq I_2 \supseteq \ldots \supseteq I_t \supseteq \ldots$ be a sequence of $m$-primary ideals such that each $R/I_t$ is Gorenstein and the $I_t$ are cofinal with the powers of $m$. Let $E=E_R(K)$, the injective hull of $K$ over $R$. Then $E$ is equal to the increasing union $\bigcup_t \Ann_E(I_t)$. Further, each $\Ann_E(I_t)$ is isomorphic to $E_{R/I_t}(K) \cong R/I_t$, so we have injective maps $R/I_t \to R/I_{t+1}$ for each $t \ge 1$. Let $u_1$ be a generator of the socle in $R/I_1$. For $t \ge 1$, let $u_{t+1}$ be the image of $u_t$ in $R/I_{t+1}$, which will generate the socle in $R/I_{t+1}$.

Suppose that $M$ is a \fg\ \CM\ module with no free summand. We will show that $M$ is equal to the increasing union of $I_tM:u_t$, so that $u_tM \subseteq I_tM$ for $t \gg 1$. This will imply that $M$ gives us a nontrivial Dietz closure. To see that the union is increasing, suppose that $v \in I_tM:u_t$. Then $u_tv \in I_tM$. Applying the map $I_tM \to I_{t+1}M$ induced by the map $R/I_t \to R/I_{t+1}$, we see that $u_{t+1}v \in I_{t+1}M$.

Suppose that $M \ne \bigcup_t I_tM:u_t$. Then we can pick $v \in M- \bigcup_t I_tM:u_t$. For every $t \ge 1$, $u_tv \not\in I_tM$. Consider the map $R \to M$ given by multiplication by $v$. Since $R$ is local and $M$ is \fg, this splits if and only if $E \to E \otimes M$ is injective. But this is true if and only if $R/I_t \to M/I_tM$ is injective for all $t \gg 1$. For any $t$, $u_t \mapsto u_tv \not\in I_{t+1}M$, so the socle of $R/I_t$ is not contained in the kernel of the map $R/I_t \to M/I_tM$. Hence $R/I_t \to M/I_tM$ is injective, which implies that $R \to M$ splits. This contradicts our assumption that $M$ had no free summand.

If $R$ is not regular, then since $R$ is \CM, $\syz^d(k)$ is a \fg\ \CM\ module that is not free. Then it has some minimal direct summand (which can't be written as a nontrivial direct sum) that is not free. This gives us a nontrivial Dietz closure on $R$. Therefore, $R$ must be regular.
\end{proof}

\begin{remark}
By a result of \cite{dutta}, $\syz^d(k)$ has no free summand when $R$ is not regular, so we can use $\syz^d(k)$ instead of a minimal direct summand of it.
\end{remark}

The following is a corollary to the proof of Theorem \ref{trivialimpliesregular}.

\begin{corollary}
\label{nontrivialdietzclosure}
Let $R$ be a local domain with at least one Dietz closure. Suppose that $R$ has a \fg\ \CM\ module $B$ with no free summands and that $R$ is approximately Gorenstein but not regular. Then $R$ has a nontrivial Dietz closure, ${\cl}_B$.

$R$ satisfies these hypotheses when it is \CM, $\dim(R)\ne 1$, and $R$ is not regular.
Alternatively, it suffices for $R$ to be complete but not regular. If $R$ is \CM\ of dimension not equal to 1 but is not regular, $B=\syz^d(k)$ gives a nontrivial closure on $R$. In particular, if $R$ has equal characteristic, $\dim(R) \ne 1$, and $R$ is weakly F-regular (or F-regular or strongly F-regular) but not regular, ${\cl}_{\syz^d(k)}$ is nontrivial on $R$.
\end{corollary}

\section{Proofs that certain closures are not Dietz closures}

Dietz gives some examples of Dietz closures, as well as some closures that fail to be Dietz closures. Understanding why certain closure operations fail to be Dietz closures adds to our understanding of Dietz closures, and may help us find a good closure operation for rings of mixed characteristic. The following result gives one way for a closure operation to be ``too big" to be a Dietz closure.

%

\begin{thm}
\label{notdietz}
Let $R$ be a local domain with $x_1,\ldots,x_k$ part of a \sop\ for $R$ and $(x_1 \cdots x_k)^t \in (x_1^{t+1},x_2^{t+1},\ldots,x_k^{t+1})^{\cl}$ for some $t \ge 0$ and closure operation cl. Then cl is not a Dietz closure.
\end{thm}

\begin{proof}
Suppose that cl is a Dietz closure. Then by Theorem \ref{dietzinbigcm}, there is a big \CM\ module $B$ such that ${\cl} \subseteq {\cl}_B$. Then we have
\[(x_1 \cdots x_k)^t \in (x_1^{t+1},\ldots,x_k^{t+1})^{\cl} \subseteq (x_1^{t+1},\ldots,x_k^{t+1})^{{\cl}_B}.\]
By Proposition \ref{strongccprop}, this implies that
\[(x_2 \cdots x_k)^t \in (x_1,x_2^{t+1},\ldots,x_k^{t+1})^{{\cl}_B},\]
which implies that
\[(x_3 \cdots x_k)^t \in (x_1,x_2,x_3^{t+1},\ldots,x_k^{t+1})^{{\cl}_B},\]
and so on until
\[1 \in (x_1,\ldots,x_k)^{{\cl}_B}.\]
But $(x_1,\ldots,x_k)^{{\cl}_B} \subseteq m^{{\cl}_B}=m$. As $B \not\subseteq (x_1,\ldots,x_k)B$, this is a contradiction.
Therefore, cl is not a Dietz closure.
\end{proof}

\begin{corollary}
Integral closure is not a Dietz closure on $R$ if $\dim(R) \ge 2$.
\end{corollary}

\begin{proof}
Let $x,y$ be part of a \sop\ for $R$. We always have $xy \in \overline{(x^2,y^2)}$, so by Theorem \ref{notdietz}, integral closure is not a Dietz closure.
\end{proof}

%

\begin{citeddefn}[\cite{regularclosure}]
We define \textit{regular closure} on a ring $R$ by $u \in N_M^{\reg}$ if for every regular $R$-algebra $S$, $u \in N_M^{{\cl}_S}$, where $N \subseteq M$ are \fg\ $R$-modules and $u \in M$.
\end{citeddefn}

\begin{lemma}
\label{forregclosure}
Let $R=k[[x,y,z]]/(x^3+y^3+z^3)$, where $\text{char}(k) \ne 3.$ Then $(x,y)^t \subseteq (x^t,y^t)^\reg$.
\end{lemma}

\begin{proof}
In \cite{phantomhomology}, Hochster and Huneke show that $z \in (x,y)^\reg$ but $z \not\in (x,y)^*$. To do this, they reduce to the case of maps to complete regular local rings with algebraically closed residue field and show that any solution $(a,b,c)$ of $u^3+v^3+w^3=0$ in $S$ has the form $(\alpha d,\beta d,\gamma d)$, where $d \in S$ and $(\alpha,\beta,\gamma)$ is a solution of the same equation such that either at least two of $\alpha,\beta,$ and $\gamma$ are units, or all three are 0.

If all three are 0, then clearly $(x,y)^tS \subseteq (x^t,y^t)S$. If at least one of $\alpha$ or $\beta$ is a unit, then $(x^t,y^t)S=(d^t)$, which must contain $(x,y)^tS$. Since these are the only possible cases, we have $(x,y)^tS \subseteq (x^t,y^t)S$ for any regular $R$-algebra $S$. Hence $(x,y)^t \subseteq (x^t,y^t)^\reg$.
\end{proof}

\begin{corollary}
Regular closure may fail to be a Dietz closure.
\end{corollary}

\begin{proof}
Consider the ring  $R=k[[x,y,z]]/(x^3+y^3+z^3)$, where $\text{char}(k) \ne 3.$ In this ring, $xy \in (x^2,y^2)^\reg$ by Lemma \ref{forregclosure}.
\end{proof}

\begin{remark}
By the exact argument used in Lemma \ref{forregclosure}, UFD closure (consider all $R$-algebras that are UFD's, rather than the regular $R$-algebras) may fail to be a Dietz closure.
\end{remark}

\begin{thm}
\label{solidclosurechar0}
For rings of equal characteristic 0, solid closure is not always a Dietz closure. In particular, solid closure is not a Dietz closure on regular local rings containing the rationals with dimension at least 3.
\end{thm}

\begin{proof}
By Theorem \ref{regularimpliestrivial}, Dietz closures are trivial on regular rings. By \cite[Corollary 7.24]{solidclosure} and \cite{solidclosurenontrivialcomputation}, if $R$ is a regular local ring containing the rationals with dimension at least 3, then solid closure is not trivial on $R$. Hence solid closure is not a Dietz closure on these rings. 
\end{proof}

\section{Full Extended Plus Closure}

We do not know whether Heitmann's mixed characteristic plus closure, full extended plus closure, and full rank one closure \cite{heitmann} are Dietz closures, even in dimension 3. To discuss this question, we first extend the definition of full extended plus closure (epf) to finitely generated modules. The other definitions can be extended similarly.

\begin{defn}
Let $R$ be a mixed characteristic local domain, whose residue field has characteristic $p$. Let $N \subseteq M$ be finitely generated modules over $R$. We define the \textit{full extended plus closure} of $N$ in $M$ by $u \in M$ is in $N_M^{\epf}$ if there is some $c \ne 0 \in R$ such that for all $n \in \Z_+$, 
\[c^{1/n} \otimes u \in \im(R^+ \otimes N + R^+ \otimes p^nM \to R^+ \otimes M).\]
\end{defn}

\begin{prop}
For $R$ a local domain of mixed characteristic $p$, full extended plus closure is a closure operation that satisfies the \axioma, the \axiomb, and the \axiomc, and $0^\epf_R=0$.
\end{prop}

\begin{proof}
It is easy to prove the extension and order-preservation properties.
To see that epf is idempotent, making it a closure operation, let $u \in (N^\epf_M)^\epf_M$. Then there is some $c \ne 0$ in $R$ such that 
\[c^{1/n} \otimes u \in \im(R^+ \otimes N^\epf_M + R^+ \otimes p^nM \to R^+ \otimes M)\]
 for all $n$, say
\[c^{1/n} \otimes u = \sum_i r_i \otimes y_i+\sum_j s_j \otimes p^nm_j,\]
with $r_i,s_j \in R^+$, $y_i \in N^\epf_M$, and $m_j \in M$. For each $i$, there is some nonzero $d_i \in R$ such that 
\[d_i^{1/n} \otimes y_i \in \im(R^+ \otimes N+R^+ \otimes p^nM \to R^+ \otimes M).\]
Then 
\[
c^{1/n}\cdot \Pi_i d_i^{1/n} \otimes u =\Pi_i d_i^{1/n}\left(\sum_i r_i \otimes y_i+\sum_j s_j \otimes p^nm_j \right) 
\in \im(R^+ \otimes N + R^+ \otimes p^nM \to R^+ \otimes M). 
\]
Since $c \cdot \Pi_i d_i$ is a nonzero element of $R$, this proves that $u \in N^\epf_M$.

For the \axioma, let $f:M \to W$ be an $R$-module homomorphism and $N \subseteq M$. Let $u \in N^\epf_M$. Then there is some nonzero $c \in R$ such that 
\[c^{1/n} \otimes u \in \im(R^+ \otimes N+R^+ \otimes p^nM \to R^+ \otimes M)\]
for every $n>0$.
Apply $f$. This tells us that 
\[c^{1/n} \otimes f(u) \in \im(R^+ \otimes f(N)+R^+ \otimes p^nW \to R^+ \otimes W)\]
for every $n>0$, which implies that $f(u) \in f(N)^\epf_W$.

Next, suppose that $N^\epf_M=N$. We will show that $0^\epf_{M/N}=0$. Let $\bar{u} \in 0^\epf_{M/N}$, where $u \in M$. Then there is some nonzero $c \in R$ with
\[ c^{1/n} \otimes \bar{u} \in \im(R^+ \otimes p^n(M/N) \to R^+ \otimes M).\]
But $R^+ \otimes p^n(M/N)$ is isomorphic to $p^n(R^+ \otimes M)/(R^+ \otimes N)$, which tells us that
\[c^{1/n} \otimes u \in \im(R^+ \otimes p^nM+R^+ \otimes N \to R^+ \otimes M).\]
This implies that $u \in N^\epf_M=N$, so $\bar{u}=0$ in $M/N$.

To see that $m^\epf_R=m$, let $u \in m_R^\epf$. Then
\[c^{1/n}u \in (m,p^n)R^+\]
for some nonzero $c \in R$ (using the ideal version of the definition of epf) and for all $n$. Since $p^n \in m$, $c^{1/n}u \in mR^+$ for all $n$. If $u \not\in m$, then $c^{1/n} \in mR^+$ for all $n$. But we can extend the $m$-adic valuation on $R$ to a $\Q$-valued valuation on $R^+$. The order of $c^{1/n}$ will be $\frac{1}{n}\ord(c)$. So this is impossible.

Now let $u \in 0^\epf_R$. Then $c^{1/n}u \in p^nR^+$ for some $c \ne 0$ in $R$ and for all $n$. Let ord denote a $\Q$-valued valuation on $R^+$ that extends the $m$-adic valuation on $R$. Let $s=\ord(c)$ and $t=\ord(p)$. Then we must have $s/n+\ord(u) \ge nt$ for all $n$. This implies that $u=0$.
\end{proof}

A similar argument works for mixed characteristic plus closure and for full rank one closure.

If at least one of these closures is a Dietz closure in dimension 3, this would tie the results of \cite{heitmann, dim3bigcmalgebras} in to the results of this paper. If they are not Dietz closures in dimension 3, this would imply that the Dietz axioms are stronger than they need to be--there could be a weaker set of axioms that would be sufficient for the proof of the Direct Summand Conjecture in mixed characteristic rings.

\section{Connections between Dietz closures and other closure operations}
\label{otherclosureops}

We show that Dietz closures are contained in (liftable) integral closure. This is proved for ideals in \cite{dietzthesis} with the added assumption that the closures are persistent for change of rings, but we do not need this assumption here.

\begin{thm}
\label{solidsmall}
Let $R$ be a domain and ${\cl}={\cl}_M$ where $M$ is a solid module over $R$. Then $I^{\cl} \subseteq \bar{I}$ for every ideal $I$ of $R$.
\end{thm}

\begin{proof}
Since $M$ is solid, there is some nonzero map $f:M \to R$, with image $\frak{a}$, a nonzero ideal of $R$. Suppose that $I \subseteq J \subseteq I^{\cl}$. Then $JM=IM$. Applying $f$, we get $J\frak{a}=I\frak{a}$. Since $R$ is a domain, $\frak{a}$ is a \fg, torsion-free $R$-module. By the lemma below, $J \subseteq \bar{I}$.
\end{proof}

\begin{citedlemma}[\cite{integralclosure}]
Suppose that $I \subseteq J$ are ideals of a domain $R$ such that $IM=JM$ for some \fg, torsion-free $R$-module $M$. Then $J \subseteq \bar{I}$.
\end{citedlemma}

\begin{corollary}
Let $R$ be a complete local domain and $B$ a big \CM\ module over $R$. Then $I^{{\cl}_B} \subseteq \bar{I}$ for all ideals $I$ of $R$.
\end{corollary}

\begin{proof}
By \cite[Proposition 10.5]{solidclosure}, $B$ is a solid module over $R$. Hence by Theorem \ref{solidsmall}, $I^{{\cl}_B} \subseteq \bar{I}$ for every ideal $I$ of $R$.
\end{proof}

There are several ways to extend integral closure to modules. Here we use liftable integral closure, denoted $\vdash$, as defined by Epstein and Ulrich.

\begin{citeddefn}[{\cite{liftableintegralclosure}}]
Let $G$ be a submodule of a \fg\ free $R$-module $R^s$, let $S$ be the symmetric algebra over $R$ defined by $R^s$, and let $T$ be the subring of $S$ induced by the inclusion $G \subseteq R^s$. Observe that $S$ is $\N$-graded and generated in degree 1 over $R$, and that $T$ is an $\N$-graded subring of $S$, also generated in degree 1 over $R$. We define the \textit{integral closure} $G^-_{R^s}$ of $G$ in $R^s$ to be the degree 1 part of the integral closure of the subring $T$ of $S$.

Now let $N \subseteq M$ be \fg\ $R$-modules. Take a free module $R^s$ and a surjection $\pi:R^s \to M$, and let $G=\pi^{-1}(N)$. We define the \textit{liftable integral closure} of $N$ in $M$ by 
\[N_M^\vdash=\pi(G_{R^s}^-).\]
\end{citeddefn}

\begin{prop}
\label{solidsmallformodules}
Let $R$ be a domain and ${\cl}={\cl}_M$ where $M$ is a solid $R$-module. Then for all \fg\ free modules $F$ over $R$ and all submodules $G$ of $F$, $G^{\cl}_F \subseteq G^{\vdash}_F$. 
\end{prop}

\begin{proof}
Let $F$ be a free module of rank $h$ over $R$ and $G \subseteq F$. Let $S=\symm(F) \cong R[x_1,\ldots,x_h]$, $I$ the ideal generated by the image of $G$ in $S$, and $\widetilde{M}=S \otimes_R M$. We will show that $G^{\cl}_F$ is contained in the degree one piece of $I^{{\cl}_{\widetilde{M}}}_S$.

Suppose that $u \in G^{\cl}_F$. Then for every $m \in M$, $m \otimes u \in \im(M \otimes G \to M^h)$. This implies that $m \otimes u \otimes 1 \in \im(M \otimes_R G \otimes_R S \to M^h \otimes_R S).$ By associativity and commutativity of tensor, $M \otimes_R G \otimes_R S \cong \widetilde{M} \otimes_S I \cong I\widetilde{M}$. This isomorphism takes $m \otimes u \otimes 1 \mapsto u(1 \otimes m)$. Then $u(s \otimes m) \in I\widetilde{M}$ for all $s \in S$, $m \in M$, which implies that $u \in I^{{\cl}_{\widetilde{M}}}_S$. Since $u \in G$, its image in $S$ is of degree 1.

Since $S$ is a domain and $\widetilde{M}$ is solid over $S$, $I^{{\cl}_{\widetilde{M}}} \subseteq \bar{I}$ for all ideals $I$ of $S$. This implies that $u$ is contained in the degree 1 piece of $\bar{I}$, and hence $u \in G^{\vdash}_F$.
\end{proof}

\begin{thm}
Let $R$ be a domain and ${\cl}={\cl}_M$ where $M$ is a solid module over $R$. Then ${\cl}$ is contained in liftable integral closure. In particular, if $R$ is a complete local domain, all big \CM\ modules closures on $R$ are contained in liftable integral closure. This implies that all Dietz closures on $R$ are contained in liftable integral closure.
\end{thm}

\begin{proof}
Let $L \subseteq N$ be \fg\ modules over $R$, and let $\pi:F \to N$ be a surjection of a \fg\ free module $F$ onto $M$. Let $K=\pi^{-1}(L)$. Let $u \in L^{\cl}_N$. Then by Lemma \ref{lemma1.2}, any lift $\tilde{u}$ of $u$ to $F$ is contained in $K^{\cl}_F$. By Proposition \ref{solidsmallformodules},  $\tilde{u} \in K^{\vdash}_F$. Hence $u \in L^{\vdash}_N$.
\end{proof}

Recall that a family of closure operations cl on a class of rings and maps between them is persistent  for change of rings if given any $R \to S$ in the class, and $N \subseteq M$ \fg\ $R$-modules,
\[S \otimes_R N_M^{\cl} \subseteq (S \otimes_R N)^{\cl}_{S \otimes_R M}.\]

\begin{prop}
Any persistent family of Dietz closures is contained in regular closure.
\end{prop}

\begin{proof}
Suppose that $u \in I^{\cl}$. Then in any map to a regular ring $S$, $u \in (IS)^{\cl}=IS$ by persistence. So $u \in I^\reg$.
\end{proof}

\section{Further Questions}
\label{furtherquestions}

\subsection{Examples of \CM\ Module Closures}

In the proof of Theorem \ref{trivialimpliesregular}, we showed that if a local domain $(R,m,k)$ is \CM\ of dimension not equal to 1 but is not regular, ${\cl}_{\syz^d(k)}$ is a non-trivial Dietz closure for $R$. We give another class of non-trivial Dietz closures, which can only occur when $R$ is not regular.

\begin{eg}
Let $R=k[[x^2,xy,y^2]]$. Then $M=(x^2,xy)$ is a non-maximal \CM\ module over $R$ (it has height=depth=1). Let $I=(x^4,x^3y,xy^3,y^4)$ and $J=(x^4,x^3y,x^2y^2,xy^3,y^4)$. Then $I \subsetneqq J$, but 
\[I(x^2,xy)=(x^6,x^5y,x^3y^3,x^2y^4,x^5y,x^4y^2,x^2y^4,xy^5)=J(x^2,xy).\]
So $I^{{\cl}_M}=J^{{\cl}_M}$.
\end{eg}

\begin{eg}
Let $R=k[[x,y,u,v]]/(xy-uv)$. Then $M=(x,u)$ is a non-maximal \CM\ module over $R$. Let $I=(y^2,v^2)$ and $J=(yv)$. Then $I \ne J$, but $IM=JM=(xyv,yuv)$, so $I^{{\cl}_M}=J^{{\cl}_M}$.

In addition, if we let $I=(x^2,u^2)$ and $J=(x^2,xu,u^2)$, then $IM=JM=(x^3,x^2u,xu^2,u^3)$, even though $I \ne J$.
\end{eg}

This gives rise to a more general class of examples: suppose that $(x,u)$ is a non-principal ideal that is a \CM\ module (height 1, depth 1), and $xu \not\in (x^2,u^2)$. Then $(x^2,u^2)(x,u)=(x^2,xu,u^2)(x,u)$.

\begin{eg}
Let $R=k[[x,y,z]]/(x^3+y^3+z^3)$, with the characteristic of $k$ not equal to 3. Then $(x,y+z)$ is a height 1 prime of depth 1. Since $x(y+z) \not\in (x,y+z)^2$, we are in the case above.
\end{eg}

All of these examples are Gorenstein rings, so in each case the canonical module (a maximal \CM\ module) is equal to the ring.

\begin{question}
If $R$ is not Gorenstein and has a canonical module $\omega$, then $\omega$ is a \CM\ module for $R$ with no free summand. Hence by the proof of Theorem \ref{trivialimpliesregular}, ${\cl}_\omega$ is a nontrivial closure operation on $R$. How else might we characterize this closure?
\end{question}

\subsection{Largest Big \CM\ Module Closure}
\label{largestbigcm}

We do not know whether there is a largest Dietz closure. If there is one, then by Theorem \ref{dietzinbigcm} it will also be the largest big \CM\ module closure. Hence there is a largest big \CM\ module closure if and only if there is a largest Dietz closure.

\begin{prop}
If Dietz closures on a local domain $R$ form a directed set, then the sum of all Dietz closures is equal to the largest Dietz closure.
\end{prop}

\begin{proof}
Let D denote the sum of all Dietz closures. To see that it is a Dietz closure (it will be a closure operation, by \cite{epstein}), we use the fact \cite{epstein} that since $R$ is Noetherian, for any particular $N \subseteq M$ \fg\ $R$-modules, there is some Dietz closure cl such that $N_M^{\cl}=N_M^\ld$.

\axioma: Let $f:M \to W$ be a map of $R$-modules, and $N \subseteq M$. Let cl be a Dietz closure such that $N_M^{\cl}=N_M^\ld$. Then $f(N_M^\ld)=f(N_M^{\cl})\subseteq f(N)_W^{\cl} \subseteq f(N)_W^\ld$.

\axiomb: Suppose that $N_M^\ld=N$. Then $N_M^{\cl}=N$ for every Dietz closure cl. Hence $0_{M/N}^{\cl}=0$ for every Dietz closure cl, which implies that $0_{M/N}^\ld=0$.

\axiomc: We must have $m^\ld=m$, 
since $m$ 
is cl-closed for any Dietz closure cl.

\axiomd: With $R$, $v$, and $J$ as in the statement of Axiom 4, let cl be a Dietz closure such that $(Rv)_M^\ld=(Rv)_M^{\cl}$. Then $(Rv)_M^\ld \cap \ker(f)=(Rv)_M^{\cl} \cap \ker(f) \subseteq (Jv)_M^{\cl} \subseteq (Jv)_M^\ld$.
\end{proof}

So to prove that there is a largest Dietz closure, it suffices to show that Dietz closures form a directed set. To do this, it would be enough to show that given 2 Dietz closures ${\cl}$ and ${\cl}^\prime$, we can construct a big \CM\ module $B$ such that ${\cl},{\cl}^\prime \subseteq {\cl}_B$. It is not clear that if we perform a modification that is cl-phantom, then one that is ${\cl}^\prime$-phantom, that $\im(1)$ stays out of the image of $m$, so we do not know of a way to construct such a big \CM\ module.

\section*{Acknowledgments}

My thanks to Mel Hochster for many helpful conversations relating to this work, to Anurag K. Singh for suggesting Theorem \ref{solidclosurechar0}, and to Craig Huneke for asking several relevant questions. I also thank Geoffrey D. Dietz, Linquan Ma, and Felipe Perez for their comments on a draft of the manuscript, and the anonymous referee for many suggestions that improved the exposition of this paper. This work was partially supported by the National Science Foundation [grant numbers DGE 1256260, DMS 1401384]. 

\bibliography{rrgbibfile}{}
\bibliographystyle{alpha}
\end{document}